\documentclass[]{amsart}
\usepackage[utf8]{inputenc}  
\usepackage{csquotes}
\usepackage[T1]{fontenc}
\usepackage{graphicx}
\usepackage{tikz}
\usetikzlibrary{positioning}
\usepackage{float}
\usepackage{algorithm}
\usepackage{algorithmic}
\usepackage[utf8]{inputenc}
\usepackage[export]{adjustbox}
\usepackage{wrapfig}
\usepackage{amsfonts}
\usepackage{amsthm}
\usepackage{array}
\usepackage[top=2cm, bottom=2cm, left=2cm, right=2cm]{geometry}
\usepackage{amssymb}
\usepackage{xcolor}
\usepackage{mathtools}
\usepackage{comment}
\usepackage[main=english]{babel}
\usepackage{subcaption}
\usepackage{dsfont}
\usepackage{bbm}
\usepackage{lmodern}
\usepackage{indentfirst}
\usepackage{tabto}
\usepackage{color}
\usepackage{systeme}
\usepackage{eso-pic}
\usepackage{float}
\usepackage{fancyhdr}
\usepackage{appendix}
\usepackage{listings}
\usepackage{pgf, tikz}
\usepackage[utf8]{inputenc}
\usepackage{lastpage}

\definecolor{bf}{rgb}{0,0,0.6} 
\definecolor{mygray}{gray}{0.85}
\definecolor{darkWhite}{rgb}{0.94,0.94,0.94}

\graphicspath{{images/}}

\newcommand{\espcond}[2]{\mathbb{E}\mathopen{}\left[#1\middle|#2\right]}
\newcommand{\esp}[1]{\mathbb{E}\mathopen{}\left[#1\right]}

\newcommand{\reels}{\mathbb{R}}

\usepackage{thmtools,thm-restate}

\usepackage{hyperref}

\numberwithin{equation}{section}

\usepackage{amsthm, amssymb}

\newtheorem{thm}{Theorem}[section]
\newtheorem{definition}[thm]{Definition}

\newtheorem{lemma}[thm]{Lemma}
\newtheorem{hyp}[thm]{Hypothesis}
\newtheorem{corol}[thm]{Corollary}
\theoremstyle{definition}
\newtheorem{remarque}[thm]{Remark}

\title{Learning controls in mean field games: the fictitious play and related gradient descent }
\begin{document}

\begin{abstract}
    We investigate two learning procedures and their relationship to the formation of equilibrium in potential mean field games (MFGs). The first one is a version of the fictitious play in which players observe directly the actions (controls) of other players and learn the distribution of the MFG through an update rules on controls. This leads to a new iterative procedure which converges for a wide class of potential MFGs of controls, even in the presence of common noise. On the other hand we also consider the learning procedure of a principal agent trying to solve a mean field control problem with only local gradient information. We prove that under very general assumption, the gradient descent of such agent converges to a solution of the associated potential MFG.
\end{abstract}
\author{Meynard Charles}
\maketitle
\section{Introduction}
\subsection{General introduction on mean field games}
Mean field games were introduced simultaneously by Lasry and Lions \cite{LasryLionsMFG, Lions-college} on the one side and Huang, Malhamé, and Caines \cite{malhame} on the other, motivated by applications in economics and game theory. This theory provides a rigorous framework for the analysis of a population of rational agents interacting with each other only through the distribution of the overall population. Given a Lagrangian $L:[0,T]\times \reels^{2d}\times \mathcal{P}_2(\reels^{2d})\to \reels$, a terminal condition $g:\reels^d\times \mathcal{P}_2(\reels^d)\to \reels$, and initial conddition $X_0\in \mathcal{H}^d$, we introduce the following cost function 
\begin{gather}
\label{eq: cost function}
    J(\alpha,(m_t)_{t\in [0,T]})=\esp{G(X^\alpha_T,\pi_d m_T)+\int_0^T L(t,X^\alpha_t,\alpha_t,m_t)dt},\\
   \nonumber X^\alpha_t=X_0-\int_0^t \alpha_s ds+\sigma W_t+\sigma^0W^0_t,
\end{gather}
where $\pi_dm$ indicates the marginal of a measure $m\in \mathcal{P}_2(\reels^{2d})$ over its first $d$ arguments, $(W_t)_{t\geq 0}$ is a Brownian motion describing the idiosyncractic noise and $(W^0_t)_{t\geq 0}$ describes the evolution of the common noise. In this article, we are interested in the associated mean field game: the functional $J$ indicates the cost a single player would pay if their initial position was $X_0$, their control $(\alpha_t)_{t\in [0,T]}$ and the distribution of all other players and of their controls given a realisation of the common noise $\omega^0$ was the flow $(m_t(\omega^0))_{t\in [0,T]}$. Given an initial distribution $\mu$ for the initial position of players the mean field game problem consists in finding a Nash equilibrium, that is a distribution $(m^*_t)_{t\in [0,T]}$ such that if all players expect the distribution of other players (and their controls) to be given by $(m^*_t)_{t\in [0,T]}$ and control their state optimally given this prior, then the distribution of all players at all time is indeed given by $(m^*_t)_{t\in [0,T]}$. Choosing $X_0\sim \mu$ to express this condition mathematically in terms of a representative player, this problem is equivalent to finding a control process $(\alpha^*_t)_{t\in [0,T]}$ satisfying the equilibrium condition 
\begin{gather}
    \label{eq: mfg equilibrium condition}
    (\alpha^*_t)_{t\in [0,T]}\in \underset{\alpha}{\text{arginf}}  J\left(\alpha,\left(m^*_t\right)_{t\in [0,T]}\right)\\
   \nonumber \forall t\in [0,T],\quad m^*_t=\mathcal{L}(X^{\alpha^*}_t,\alpha^*_t|\mathcal{F}^0_t),
\end{gather}
where $(\mathcal{F}^0_t)_{t\geq 0}$ indicates the filtration associated to the common noise. The equilibrium condition of MFGs can be expressed through several formulations of the problem: from a coupled system consisting of a forward equation for the distribution of players over time and a backward equation for the controls of agents (this includes both the MFG system from the HJB formulation \cite{LasryLionsMFG} and the associated FBSDE obtained from Pontryagin \cite{carmona2018probabilistic}), to a master equation encompassing the whole dynamics of the game in a single infinite-dimensional equation \cite{Lions-college, CDLL}. Many articles have already studied these different equations, in particular in several monotone settings \cite{LasryLionsMFG, meszaros} in which uniqueness of the equilibrium \cite{Bertucci-monotone, meynard-3} is recovered thanks to the adversarial nature of players. In this article we consider instead the potential regime \cite{CecchinDelarue2026}, meaning coefficients are not assumed to be monotone but the MFG problem derives from a mean field control one \cite{bensoussan2013mean}. Mathematically, this means there exists a Lagrangian $l:[0,T]\times \mathcal{P}_2(\reels^{2d})\to \reels$ and a terminal cost $g:[0,T]\times \mathcal{P}_2(\reels^d)\to \reels$, such that the following assumption is verified
\[G=\frac{\delta g}{\delta \mu}, \quad L=\frac{\delta l}{\delta m}.\]
Consequently, the mean field game admits the following potential 
\begin{equation}
    \label{eq: potential of the mfg}
    \mathcal{J}(\alpha)=\esp{\int_0^T l(t,\mathcal{L}(X^\alpha_t,\alpha_t|\mathcal{F}^0_t))dt+g(\mathcal{L}(X^\alpha_T|\mathcal{F}^0_T))}.\end{equation}
While our main goal is to consider 'true' mean field games of controls (by which we mean MFG with a non linear term depending on the distribution of controls), this also includes more standard control problems. For example taking 
\[l(m)=\int_{(x,a)\in \reels^{2d}} |a|^2 m(dx,da),\]
is equivalent to considering a quadratic penality in controls. More generally, any Lagrangian $l$ which depends linearly on its marginal over the last $d-$variables (that is, the one associated with controls) correspond to a standard mean field control problem. 
\subsection{The fictitious play}
From a practical perspective, a key question in MFGs concerns the formation of MFG equilibria, namely, how these equilibria actually arise from a disordered crowd of agents who do not have any a priori knowledge of the state of the population. Answering this question is the objective of fictitious play, introduced in the context of mean field games in \cite{cardaliaguet-fictitious-play}. In this setting, agents repeatedly play the same game, trying to infer the future actions of the other players from past experience. The question of learning is a pivotal one in game theory \cite{fudenberg1998theory}, and fictitious play has a long history \cite{brown1951iterative,Monderer_Shapley_1996,Robinson_1951}. In particular, in this framework, agents usually learn by playing the best response to the averaged controls of their counterparts over previous games, in contrast to \cite{cardaliaguet-fictitious-play}, in which agents learn directly from the distribution of the others.. This is precisely the procedure we study in this article, that is: given an initial control $(\bar{\alpha}^0_t)_{t\in [0,T]}$, we investigate the sequence of controls $(\bar{\alpha}^n)_{n\in \mathbb{N}}$ generated by the following fictitious play procedure
\begin{equation}
    \label{eq: fictitious play}
    \forall n\in \mathbb{N}\quad \left\{
\begin{array}{cc}
    \bar{m}^n_t=\mathcal{L}(X^{\bar{\alpha}^n}_t,\bar{\alpha}^n_t|\mathcal{F}^0_t)\\
    \alpha^n\in \text{arginf}J(\cdot, (\bar{m}^n_t)_{t\in [0,T]})\\
    \bar{\alpha}^{n+1}=(1-\eta_n)\bar{\alpha}^n+\eta_n\alpha^n
\end{array}\right.
\end{equation}
with $(\eta_n)_{n\in \mathbb{N}}\subset [0,1]$ a non increasing sequence of deterministic steps. This last sequence represents the learning rate of players. The interpretation of this iterative scheme is then quite straightforward: at each iteration players choose the best response $\alpha^n$ to their belief $(\bar{m}^n_t)_{t\in [0,T]}$ on the distribution of other players. Then they update their belief on the control played by the population through
\[\bar{\alpha}^{n+1}=(1-\eta_n)\bar{\alpha}^n+\eta_n\alpha^n.\]
At this point, this update rule directly on controls may seem slightly surprising, we come back later, in section \ref{section: interpretation}, to its interpretation and how it is perfectly natural for learning in MFGs.

This new update rule for the fictitious play converges under a different set of assumptions and can, in particular, handle MFGs of controls \cite{cardaliaguet_Lehalle_2018}, that is, MFGs in which the cost of a single agent depends on the distribution of both the states and actions (controls) of the other players. In our opinion, both this update rule on controls and the one proposed in \cite{cardaliaguet-fictitious-play} directly on probability measures are equally natural, and we emphasize in this article how our main convergence results relate to the formation of equilibria in MFGs. The fact that two possible learning rules coexists is in our opinion, another manifestation of the two known approaches in mean field games: the probabilistic theory \cite{carmona2018probabilistic} relying on stochastic processes and a more analytic point of view which consists in working directly at the level of probability measures \cite{cardaliaguet-fictitious-play}.

\subsection{main results and related litterature}
Our main result concerns the convergence of the fictitious play \eqref{eq: fictitious play} to the set of MFG equilibria satisfying \eqref{eq: mfg equilibrium condition} in an apropriate weak sense due to the presence of common noise. To the author's knowledge this is the first instance of a fictitious play method which can handle general MFGs in which the dependency on the measure arguments and on the control of a single player are not separated. Another novelty of the approach is the fact that we can tackle MFG problems on whole space in a potential, non monotone regime, whereas previous results in this class were restricted to study on the torus \cite{cardaliaguet-fictitious-play}. 

Since the initial findings of \cite{cardaliaguet-fictitious-play}, many articles have built upon these results in several directions. The first focuses on the convergence rate of fictitious play and how to improve it by controlling the learning rate of players \cite{Lavigne_Pfeiffer_2023}. A second direction is the application of this learning procedure to reinforcement learning in MFGs \cite{Elie_Pérolat_Laurière_Geist_Pietquin_2020,Perrin2020Fictitious}. The new learning procedure proposed here could prove useful in this direction, although we do not make any quantitative claims. Finally, a third direction of research concerns the development of fictitious play with common noise \cite{carmona2018probabilistic}. In this case, the lack of compactness due to the presence of such noise makes the analysis difficult. By relying on the theory of relaxed controls in MFGs \cite{djete-measure-valued,Djete_2023,lacker}, we prove that the fictitious play algorithm proposed in this paper converges to an MFG equilibrium (in some relaxed sense), even in the presence of common noise. To the extent of the authors' knowledge, this is the first convergence result for fictitious play with common noise in MFGs outside of a regime in which coefficients are monotone with respect to the measure argument considered in \cite{Perrin2020Fictitious}.

Let us finally comment on some interesting developments in fictitious play for mean field games that we do not address in this article. First, the convergence of this procedure when the state space of players is discretized was treated in \cite{Hadikhanloo_Silva}, which is of interest for both modeling and numerical concerns. Let us also mention the study of adjacent algorithms such as mirror descent, as explored in \cite{hadikhanloo-sorin}. None of these problems have been investigated in the context of the fictitious play proposed in this article or, more generally, for mean field games of controls, and they represent interesting future directions of study.

A second part of this article is also dedicated to explaining how gradient descent in mean field control (MFC) problems converges under very general assumptions to a MFG equilibrium of the associated potential game. It is well known that in this case, the MFG corresponds to the first-order optimality conditions for the MFC \cite{carmona2018probabilistic}, and we prove that a principal agent looking to solve the MFC problem with only local information will eventually converge to a control solution of the potential MFG (which is not necessarily optimal for the MFC). This justifies how MFG equilibria can also arise from a principal agent problem, thereby establishing a link with the formation of equilibria in MFGs.

\subsection{Technical setting and notations}
In this article, we borrow the probabilistic setting of \cite{carmona2018probabilistic} Section 7. Namely, we fix a horizon $T>0$ and consider a complete probability space $(\Omega^0,\mathcal{F}^0,\mathbb{P}^0)$ such that $(\mathcal{F}^0_t)_{0\leq t\leq T}$ is a complete and right continuous filtration, and there exists $(W^0_t)_{0\leq t\leq T}$ a $d^0-$dimensional $\mathcal{F}^0-$Brownian motion. Similarly we are given $(\Omega^1,\mathcal{F},\mathbb{P}^1)$ a complete probability space with a complete and right continuous filtration $(\mathcal{F}^1_t)_{0\leq t\leq T}$ and endowed with a $d-$dimensional $\mathcal{F}^1-$Brownian motion $(B^0_t)_{0\leq t\leq T}$. Then we work on $(\Omega,\mathcal{F},\mathbb{P})$ the completion of the product space $(\Omega^0\otimes \Omega^1,\mathcal{F}^0\otimes \mathcal{F}^1,\mathbb{P}^0\otimes \mathbb{P}^1)$ endowed with the filtration $(\mathcal{F}_t)_{0\leq t\leq T}$ obtained by augmenting the product filtration $(\mathcal{F}^0_t)_{0\leq t\leq T}\otimes(\mathcal{F}^1_t)_{0\leq t\leq T}$ to make it right continuous and completing it.
In what follows we place denote by $\mathcal{H}$ the Hilbert of square integrable random variables $L^2(\Omega,\reels)$ endowed with the canonical inner product
\[\langle \cdot,\cdot\rangle: (X,Y)\mapsto \esp{XY}.\]
Given a time horizon $T$ we may also consider the Hilbert space of progressively measurable random variables in $L^2([0,T],L^2(\Omega,\reels))$
\[\mathcal{H}^T=\left\{ (\alpha_t)_{t\in [0,T]}, \quad \mathcal{F}-\text{adapted and such that } \esp{\int_0^T |\alpha_t|^2dt}<+\infty\right\},\]
endowed with the inner product 
\[\langle\cdot,\cdot\rangle^T:(\alpha,\beta)\mapsto \esp{\int_0^T \alpha_t\beta_tdt}.\]
For the induced norm we use the notation
\[\|\alpha\|_T=\sqrt{\langle \alpha,\alpha\rangle^T}=\sqrt{\esp{\int_0^T |\alpha_t|^2dt}}.\]
We define the space of admissible stochastic flow of measure
\[\mathcal{H}_m=\left\{(m_t)_{t\in [0,T]}\subset L^1(\Omega, L^2([0,T],\mathcal{P}_2(\reels^{2d}))),\quad (m_t)_{t\in [0,T]} \text{ is }\mathcal{F}^0-\text{adapted}\right\}.\]
The following assumption, on the existence of a potential will be in force throughout the article
\begin{hyp}
    \label{hyp: potential mfg}
\begin{enumerate}
    \item[-] There exists a constant $C_b\geq 0$ such that for any $(\mu,m)\in \mathcal{P}_2(\reels^d)\times \mathcal{P}_2(\reels^{2d})$ and for any $t\in [0,T]$
    \[l(t,m),g(\mu)\geq -C_b,\]
    \item[-] for any $(t,x,\alpha,m)\in [0,T]\times \reels^d\times \reels^d\times \mathcal{P}_2(\reels^{2d})$
    \[L(t,x,\alpha,m)=\frac{\delta l}{\delta m}(t,m)(x,\alpha),\quad G(x,\mu)=\frac{\delta g}{\delta \mu}(\mu)(x).\]
\end{enumerate}
\end{hyp}
In addition to Hypothesis \ref{hyp: potential mfg}, we also make the following assumption on the regularity and boundedness of coefficients
\begin{hyp}
\label{hyp: regularite coefficient mfg}
The derivatives $\nabla_x G, \nabla_xL, \nabla_\alpha L$ exists and are continuous in all of their arguments. Moreover 
\[\mu\mapsto \nabla_x G(x,\mu), m\mapsto (\nabla_xL,\nabla_\alpha L)(t,x,\alpha,m),\]
are Lipschitz for the $\mathcal{W}_2$ distance, uniformly in the other arguments and there exists a constant $C\geq 0$  such that for any $(t,x,\alpha,m)\in [0,T]\times \reels^{2d}\times \mathcal{P}_2(\reels^{2d})$, 
\begin{equation}
    \label{eq: linear growth of coefs}
    |f(t,x,\alpha,m)|\leq C\left(1+|x|+|\alpha|+\sqrt{\int_{\reels^{2d}}|(y,a)|^2m(dy,da)}\right),\end{equation}
for $f=\nabla_x L,\nabla_\alpha L,\nabla_x G$. Moreover, there exists a constant $C$ such that $\forall (x,\alpha,\mu,m)\in \reels^{2d}\times\mathcal{P}_2(\reels^d)\times \mathcal{P}_2(\reels^{2d}),$
\begin{gather}
    \label{eq: boundedness of coef in measure at 0}
     \quad G(0_{\reels^d},\mu)+L(0,0_{\reels^{2d}},m)\leq C.\\
     \label{eq: boundedness from below of coefficients}
     -C\leq G(x,\mu), L(x,\alpha,m)
\end{gather}
\end{hyp}
Let us insist that the continuity requirements are quite mild, in particular there is no need for the derivatives $\nabla_x G,\nabla_x L,\nabla_\alpha L$ to be Lipschitz in the finite dimensional arguments. The linear growth assumption on coefficients \eqref{eq: linear growth of coefs} can also be relaxed to a more general polynomial growth, provided that a corresponding coercivity assumption is imposed on the Lagrangian L. Finally, we make the following assumption on the convexity of coefficients in $(x,\alpha)$
\begin{hyp}
    \label{hyp: convexity in x}
    \begin{enumerate}
        \item[-]The function 
\[x\mapsto G(x,\mu),\]
is convex uniformly in $\mu\in \mathcal{P}_2(\reels^d)$. 
\item[-] There exists a constant $c_L>0$ such that the function 
\[(x,\alpha)\mapsto L(t,x,\alpha,m),\]
is convex, $c_L-$strongly in $\alpha$, uniformly in $(t,m)\in [0,T]\times \mathcal{P}_2(\reels^{2d})$.
    \end{enumerate} 
\end{hyp}
Let us insist that no convexity and/or monotonicity is imposed in the measure argument. In particular this assumption is not displacement monotonicity \cite{meszaros}, only convexity in the finite dimensional arguments. Although this is still a somewhat strong assumption, this does allow us to treat the case of a fully degenerate volatility $\sigma=0$ without any additional cost.
\begin{remarque}
We did try to relax this last assumption on the convexity of coefficients with respect to $x$. Although it is possible to consider sufficiently small convexity default with respect to the horizon $T$, or the so called "$\lambda-$convex" setting which derives easily from the convex setting by using the equality
\[\lambda |X^\alpha_T|^2=\lambda |X_0|^2-2\lambda \int_0^T X^\alpha_s\cdot\alpha_sds+ 2\lambda\sigma \int_0^T  X^\alpha_s\cdot dW_s+2\lambda \sigma^0\int_0^T X^\alpha_s\cdot dW^0_s+(\sigma^2+(\sigma^0)^2)t,\]
in which we assume instead that Hypothesis \ref{hyp: convexity in x} is satisfied for some $\lambda\geq 0$ by 
\[x\mapsto G(x,\mu)+\lambda |x|^2, (x,\alpha)\mapsto L(t,x,\alpha,m)-2\lambda x\cdot \alpha,\]
we did not manage to get rid of this assumption in general, even in the case $\sigma>0$. It is unknown to us whether this is a feature specific to the learning rule \eqref{eq: fictitious play} or if the convexity in $x$ can be relaxed by using more subtle arguments.
\end{remarque}

\section{Learning optimal controls through the fictitious play}
\subsection{Preliminary results}
Under Hypothesis \ref{hyp: regularite coefficient mfg}, coefficients are sufficiently regular so that the first order optimality conditions must be satisfied for any optimal control
\begin{lemma}
    \label{lemma: first order optimality mfg}
Given a flow of measure $(m_t)_{t\in [0,T]}\in \mathcal{H}_m$, under Hypothesis \ref{hyp: regularite coefficient mfg}, if a control $(\alpha_t)_{t\in [0,T]}\in \left(\mathcal{H}^T\right)^d$ satisfies
\[(\alpha_t)_{t\in [0,T]}\in \underset{\alpha}{\text{arginf}} J\left(\alpha,\left(m_t\right)_{t\in [0,T]}\right),\]
then 
\begin{gather}
   \label{eq: first order optimality mfg}\forall (\alpha'_t)_{t\in [0,T]}\in \left(\mathcal{H}^T\right)^d\\
    \nonumber \esp{\nabla_x G(X^\alpha_T,m_T)\cdot (X^{\alpha'}_T-X^{\alpha}_T)+\int_0^T \left(\begin{array}{c} \nabla_x L(t,X^\alpha_t,\alpha_t,m_t)\\ \nabla_\alpha L(t,X^\alpha_t,\alpha_t,m_t) \end{array}\right)\cdot \left(\begin{array}{c} X^{\alpha'}_t-X^{\alpha}_t\\ \alpha'_t-\alpha_t\end{array}\right) dt}\geq 0.\end{gather} 
\end{lemma}
\begin{proof}
By definition of the infimum, is there exists such a control $(\alpha_t)_{t\in [0,T]}$, for any other control $(\beta_t)_{t\in [0,T]}$ 
\[J(\beta)-J(\alpha)\geq 0.\]
Choosing $\beta_\lambda=\alpha+\lambda (\alpha'-\alpha)$ for some $\lambda \in [0,1]$ and $\alpha'\in (\mathcal{H}^T)^d$, we observe that 
\[\esp{G(X^{\beta_\lambda}_T,m_T)-G(X^\alpha_T,m_T)}=\esp{\int_{\kappa=0}^\lambda \nabla_x G(X^\alpha_T+\kappa (X^{\alpha'}_T-X^\alpha_T),m_T)\cdot (X^{\alpha'}_T-X^\alpha_T)d\kappa}.\]
By definition of $\alpha,\alpha'$ $(X^\alpha_T,X^{\alpha'}_T)\in \mathcal{H}^{2d}$. By the linear growth of $\nabla_x G$ following from Hypothesis \ref{hyp: regularite coefficient mfg}, \[\left(\nabla_x G(X^\alpha_T+\kappa (X^{\alpha'}_T-X^\alpha_T),m_T)\cdot (X^{\alpha'}_T-X^\alpha_T)\right)_{\kappa\in [0,1]}\] is bounded uniformly for $\kappa\in [0,1]$ in $L^1(\Omega,\reels)$. Thus, from the intermediate value theorem and the dominated convergence theorem    
\[\lim_{\lambda \to 0}\frac{1}{\lambda}\esp{G(X^{\beta_\lambda}_T,m_T)-G(X^\alpha_T,m_T)}=\esp{\nabla_x G(X^\alpha_T,m_T)\cdot (X^{\alpha'}_T-X^{\alpha}_T)}.\]
The end result follows by remarking that 
\[\lim_{\lambda\to 0}\frac{1}{\lambda}\left(J(\beta_\lambda)-J(\alpha)\right)\geq 0,\]
and similar arguments as the one we used above for the term in $g$ for the convergence of terms involving $L$.
\end{proof}

Letting $H$ indicates the Fenchel conjugate of $L$
\[\forall (t,x,p,m)\in [0,T]\times \reels^{2d}\times \mathcal{P}_2(\reels^{2d}),\quad H(t,x,p,m)=\sup_\alpha \left\{ \alpha \cdot p-L(t,x,\alpha,m)\right\},\]
if in addition coefficients are convex in the finite diemnsional arguments, then given a flow of measure, the optimal control problem of a single player is convex and admits a unique optimal control which is given explicity by the solution of a FBSDE
\begin{lemma}
    \label{lemma: characterisation of control with fbsde}
    Under Hypotheses \ref{hyp: regularite coefficient mfg} and Hypothesis \ref{hyp: convexity in x}, for any given flow of measure $(m_t)_{t\in [0,T]}\in \mathcal{H}_m$, there exists a unique optimal control $(\alpha^*_t)_{t\in [0,T]}$ solving 
    \[\alpha^*\in \underset{\alpha}{\text{arginf}} J(\alpha,(m_t)_{t\in [0,T]}).\]
    Moreover it is given by 
    \[\forall t\in [0,T], \quad \alpha^*_t=D_p H(t,X_t,U_t,m_t),\]
    where the triple $(X_t,U_t,Z_t)_{t\in [0,T]}$ is the unique strong solution to the FBSDE
    \begin{equation}
        \label{eq: control fbsde fixed measure}
        \left\{
    \begin{array}{c}
        X_t=\displaystyle X_0-\int_0^t D_p H(s,X_s,U_s,m_s)ds+\sigma W_t+\sigma^0W^0_t,\\
        U_t=\displaystyle \nabla_x G(X_T,m_T)+\int_t^T D_x H(s,X_s,U_s,m_s)-\int_t^T Z_s d(W_s,W^0_s).
    \end{array}
    \right.
\end{equation}
\end{lemma}
\begin{proof} 
\quad

\flushleft \textit{Step 1: Uniqueness and existence of an optimal control}\\
Under the present assumptions the functional $J$ is $c_L-$strongly convex on the set of controls with finite cost 
\[\left\{(\alpha_t)_{t\in [0,T]}, \mathcal{F}-\text{adapted and such that } J(\alpha)<+\infty\right\},\]
and consequently there can be at most one optimal control. In particular since for any control in this set,
\[c_L\esp{\int_0^T |\alpha_t|^2dt}\leq 2J(\alpha)+2J(0_{\mathcal{H}^d})-4J(\frac{1}{2}\alpha)<+\infty,\]
if there exists an optimal control it belongs to $\left(\mathcal{H}^T\right)^d$. Moreover by the strong convexity of $J$ and standard optimization results on reflexive Banach spaces, $J$ has a unique minimizer $\alpha^*$.
\flushleft \textit{Step 2: characterisation through the FBSDE}\\
By Lemma \ref{lemma: first order optimality mfg}, any optimal control in $\left(\mathcal{H}^T\right)^d$ must satisfy \eqref{eq: first order optimality mfg}. Introducing the BSDE 
\[U^{\alpha^*}_t=\nabla_x G(X^{\alpha^*}_T,m_T)-\int_t^T \nabla_x L(s,X^{\alpha^*}_s,\alpha^*_s,m_s)ds-\int_t^T Z^{\alpha^*}_s d(W_s,W^0_s),\]
which has a unique strong solution by standard results on BSDEs \cite{Zhang2017BSDE} Theorem 4.3.1, by Ito's Lemma, this implies that 
\[\forall \alpha'\in \left(\mathcal{H}^T\right)^d \quad \esp{\int_0^T \left((\nabla_\alpha L(t,X^{\alpha^*}_t,\alpha^*_t,m_t)-U^{\alpha^*}_t)\cdot(\alpha^*_t-\alpha'_t)\right)dt}\geq 0,\]
and then 
\[\esp{\int_0^T  \left|\alpha^*_t-D_pH(t,X^{\alpha^*}_t,U^{\alpha^*}_t,m_t)\right|^2dt}=0,\]
where we have used that $\alpha \mapsto \nabla_\alpha L(x,\alpha,m)$ is invertible of inverse $p\mapsto D_pH(x,p,m)$. From this equality in $(\mathcal{H}^T)^d$, it follows easily that the triple $(X^{\alpha^*}_t,U^{\alpha^*}_t,Z^{\alpha^*}_t)$ is a strong solution to the FBSDE \eqref{eq: control fbsde fixed measure}. Let us assume by contradiction that there exists another solution $(X_t,U_t,Z_t)_{t\in [0,T]}$ to this FBSDE, then letting 
\[\forall t\in [0,T],\quad \alpha_t=D_pH(t,X_t,U_t,m_t),\]
we observe that this control satisfies \eqref{eq: first order optimality mfg}. However, from the strong monotonicity of $\nabla_xG, \nabla_xL,\nabla_\alpha L$, there can be at most one control satisfying this inequality, which contradicts the fact that $(X_t,U_t,Z_t)_{t\in [0,T]}$ is a different solution.
\end{proof}
\subsection{Convergence of the learning procedure}
Thanks to Lemma \ref{lemma: characterisation of control with fbsde}, we now have the following equivalent representation for the fictitious play
\begin{lemma}
Under Hypotheses \ref{hyp: regularite coefficient mfg}, \ref{hyp: convexity in x}, the fictitious play procedure \eqref{eq: fictitious play} is equivalent to
\end{lemma}
\begin{equation}
    \label{eq: fbsde fictitious play}
\left\{
\begin{array}{l}
    X^n_t=\displaystyle X_0-\int_0^t D_p H(s,X^n_s,U^n_s,\mathcal{L}(\bar{X}^n_s,\bar{\alpha}^n_s|\mathcal{F}^0_s))ds+\sigma W_t+\sigma^0W^0_t,\\
    \displaystyle U^n_t=\nabla_x G(X^n_T,\mathcal{L}(\bar{X}^n_T|\mathcal{F}^0_T))+\int_t^T D_x H(s,X^n_s,U^n_s,\mathcal{L}(\bar{X}^n_s,\bar{\alpha}^n_s|\mathcal{F}^0_s))ds-\int_t^T Z^n_sd(W_s,W^0_s)\\
    \alpha^n_s=D_p H(s,X^n_s,U^n_s,\mathcal{L}(\bar{X}^n_s,\bar{\alpha}^n_s))
\end{array}
\right.
\end{equation}
\begin{gather*}\bar{\alpha}^{n+1}_t=(1-\eta_n)\bar{\alpha}^n_t+\eta_n \alpha^n_t,\\
\bar{X}^{n+1}_t=X_0-\int_0^t \bar{\alpha}^{n+1}_sds+\sigma W_t+\sigma^0 W^0_t\end{gather*}
\begin{proof}
    Given $(\bar{m}^n_t)_{t\in [0,T]}$ we can apply Lemma \ref{lemma: characterisation of control with fbsde} to deduce that there exists a unique optimal control $\alpha^n$ which is characterised by the solution of the FBSDE \eqref{eq: fbsde fictitious play}.
\end{proof}
\subsubsection{Convergence of the difference between optimal controls and averaged control}
In this first part we study the convergence of the sequence defined by $a_n=\|\alpha^n-\bar{\alpha}^n\|^2_T$. Our goal is to prove that in norm this difference tends to $0$. Obviously, if a control $\alpha$ solves the MFG, then the optimal response to the flow $\left(\mathcal{L}(X^\alpha_t,\alpha_t|\mathcal{F}^0_t) \right)_{t\in [0,T]}$ is still $\alpha$. Formally if $(a_n)_{n\in \mathbb{N}}$ tends to 0 then we expect that $(\alpha^n)_{n\in \mathbb{N}}$ should converge to a MFG equilibrium. One of the key argument of the fictitious play in potential MFGs is to use the fact that at each step, the potential $\mathcal{J}$ decreases along the sequence of controls $(\bar{\alpha}^n)_{n\in \mathbb{N}}$, 
\begin{lemma}
    \label{lemma: decreasing potential fictitious play}
    Under Hypotheses\ref{hyp: potential mfg}, \ref{hyp: regularite coefficient mfg} and \ref{hyp: convexity in x}, there exists a constant $C_{Lip}$ depending only on the Lipschitz constants of $\mu\mapsto \nabla_x G(x,\mu)$, $m\mapsto (\nabla_x L,\nabla_\alpha)(x,a,m)$ such that for any $n\in \mathbb{N}$ 
    \[\mathcal{J}(\bar{\alpha}^{n+1})-\mathcal{J}(\bar{\alpha}^n)\leq \left(-c_L \eta_n+C_{Lip}(1+T)\eta_n^2\right)\|\alpha^n-\bar{\alpha}^n\|^2_T.\]
\end{lemma}
\begin{proof}
Fixing $n\geq 0$, by definition of $g$
\begin{align*}\esp{g(\mathcal{L}(X^{\bar{\alpha}^{n+1}}_T|\mathcal{F}^0_T))-g(\mathcal{L}(X^{\bar{\alpha}^{n}}_T|\mathcal{F}^0_T))}&= \esp{\int_0^1 \nabla_x G(X^\lambda_T,\mathcal{L}(X^\lambda_T|\mathcal{F}^0_T))\cdot (X^{\bar{\alpha}^{n+1}}_T-X^{\bar{\alpha}^{n}}_T)d\lambda }\\
&=\eta_n \esp{\int_0^1 \nabla_x G(X^\lambda_T,\mathcal{L}(X^\lambda_T|\mathcal{F}^0_T))\cdot (X^{\alpha^n}_T-X^{\bar{\alpha}^n}_T)d\lambda}
 \end{align*}
with 
\[\forall \lambda \in [0,1],\quad X^\lambda_T=X^{\bar{\alpha}^{n}}_T+\lambda (X^{\bar{\alpha}^{n+1}}_T-X^{\bar{\alpha}^{n}}_T).\]
by Hypothesis \ref{hyp: regularite coefficient mfg} $\mu \mapsto \nabla_x G(x,\mu)$ is Lipschitz uniformly in $x\in \reels^d $, denoting by $C_{Lip}^g$ this Lipschitz constant we find that
\[\esp{g(\mathcal{L}(X^{\bar{\alpha}^{n+1}}_T|\mathcal{F}^0_T))-g(\mathcal{L}(X^{\bar{\alpha}^{n}}_T|\mathcal{F}^0_T))}\leq\eta_n \esp{\int_0^1 \nabla_x G(X^\lambda_T,\mathcal{L}(X^{\bar{\alpha}^n}_T|\mathcal{F}^0_T))\cdot (X^{\alpha^n}_T-X^{\bar{\alpha}^n}_T)d\lambda}+\frac{C_{Lip}^g}{2}\eta_n^2\|X^{\alpha^n}_T-X^{\bar{\alpha}^n}_T\|^2. \]
Clearly a similar computation holds for 
\[\esp{\int_0^T\left(l(\mathcal{L}(X^{\bar{\alpha}^{n+1}}_t,\bar{\alpha}^{n+1}_t|\mathcal{F}^0_t))-l(\mathcal{L}(X^{\bar{\alpha}^{n}}_t,\bar{\alpha}^n_t|\mathcal{F}^0_t))\right)dt},\]
moreover 
\[\forall t\in [0,T],\quad \|X^{\alpha^n}_t-X^{\bar{\alpha}^n}_t\|^2\leq T \|\alpha^n-\bar{\alpha}^n\|^2_T.\]
Consequently, there exists a constant $C_{Lip}$ depending only on $C^g_{Lip}$ and the Lipschitz constant of $m\mapsto (\nabla_xL,\nabla_\alpha L)(x,a,m)$ such that 
\begin{align*}\mathcal{J}(\bar{\alpha}^{n+1})-\mathcal{J}(\bar{\alpha}^n)\leq& \eta_n \esp{\int_0^1 \nabla_x G(X^\lambda_T,\mathcal{L}(X^{\bar{\alpha}^n}_T|\mathcal{F}^0_T))\cdot (X^{\alpha^n}_T-X^{\bar{\alpha}^n}_T)d\lambda} +C_{Lip}(1+T)\eta_n^2\|\alpha^n-\bar{\alpha}^n\|^2_T\\
&+\eta_n\esp{\int_0^1 \left(\begin{array}{cc} \nabla_x L\\ \nabla_\alpha L\end{array}\right)(t,X^\lambda_t,\alpha^\lambda_t,\mathcal{L}(X^{\bar{\alpha}^n}_t,\bar{\alpha}^n_t|\mathcal{F}^0_t))\cdot \left(\begin{array}{cc} X^{\alpha^n}_t-X^{\bar{\alpha}^n}_t\\ \alpha^n_t-\bar{\alpha}^n_t\end{array}\right)d\lambda},
\end{align*}
with the notation 
\[\forall (\lambda,t)\in [0,1]\times [0,T], \quad \left(\begin{array}{cc}X^\lambda_t\\ \alpha^\lambda_t\end{array}\right)=\left(\begin{array}{cc}X^{\bar{\alpha}^{n}}_t\\ \bar{\alpha}^n_t\end{array}\right)+\lambda \left(\begin{array}{cc} X^{\bar{\alpha}^{n+1}}_t-X^{\bar{\alpha}^{n}}_t\\ \bar{\alpha}^{n+1}_t-\bar{\alpha}^n_t\end{array}\right).\]
By Lemma \ref{lemma: characterisation of control with fbsde}, we remind that the BSDE
\[U_t=\nabla_x G(X^{\alpha^n}_T,\mathcal{L}(X^{\bar{\alpha}^n}_T|\mathcal{F}^0_T))-\int_t^T \nabla_x L(t,X^{\alpha^n}_t,\alpha^n_t,\mathcal{L}(X^{\bar{\alpha}^n}_t,\bar{\alpha}^n_t|\mathcal{F}^0_t))-\int_t^T Z_sd(W,W^0)_s,\]
has a unique solution which satisfies 
 \[\esp{\int_0^T \left|U_t-\nabla_\alpha L(t,X^{\alpha^n}_t,\alpha^n_t,\mathcal{L}(X^{\bar{\alpha}^n}_t,\bar{\alpha}^n_t|\mathcal{F}^0_t))\right|^2dt}=0.\]
In consequence
\begin{align*}\mathcal{J}(\bar{\alpha}^{n+1})-\mathcal{J}(\bar{\alpha}^n)\leq& \eta_n \esp{\int_0^1 (\nabla_x G(X^\lambda_T,\bar{\mu}^n_T)- \nabla_x G(X^{\alpha^n}_T,\bar{\mu}^n_T))\cdot (X^{\alpha^n}_T-X^{\bar{\alpha}^n}_T)d\lambda} +C_{Lip}(1+T)\eta_n^2\|\alpha^n-\bar{\alpha}^n\|^2_T\\
&+\eta_n\esp{\int_0^T\int_0^1 \left(\begin{array}{cc} \nabla_x L(t,X^\lambda_t,\alpha^\lambda_t,\bar{m}^n_t)-\nabla_x L(t,X^{\alpha^n}_t,\alpha^n_t,\bar{m}^n_t)\\ \nabla_\alpha L(t,X^\lambda_t,\alpha^\lambda_t,\bar{m}^n_t)-\nabla_\alpha L(t,X^{\alpha^n}_t,\alpha^n_t,\bar{m}^n_t)\end{array}\right)\cdot \left(\begin{array}{cc} X^{\alpha^n}_t-X^{\bar{\alpha}^n}_t\\ \alpha^n_t-\bar{\alpha}^n_t\end{array}\right)d\lambda dt},
\end{align*}
where we introduced the notations
\[\forall t\in [0,T],\quad \bar{m}^n_t=\mathcal{L}(X^{\bar{\alpha}^n}_t,\bar{\alpha}^n_t|\mathcal{F}^0_t),\quad \bar{\mu}^n_t=\mathcal{L}(X^{\bar{\alpha}^n}_t|\mathcal{F}^0_t).\]
Finally using the convexity of $G$ and of $L$ $c_L-$strongly in $\alpha$, it follows that
\[\mathcal{J}(\bar{\alpha}^{n+1})-\mathcal{J}(\bar{\alpha}^n)\leq \left(-c_L \eta_n(1-\frac{\eta_n}{2}) +C_{Lip}(1+T)\eta_n^2\right)\|\alpha^n-\bar{\alpha}^n\|^2_T,\]
ending the proof.
\end{proof}
Moreover under our assumptions the sequence of controls generated by the fictitious play is bounded uniformly in $\left(\mathcal{H}^T\right)^d$ for $n\in\mathbb{N}$
\begin{lemma}
    \label{lemma: boundedness in H^T, fictitious play}
    Under Hypotheses \ref{hyp: regularite coefficient mfg}, \ref{hyp: convexity in x}, there exists a constant $C_b$ depending only on $G,L$ such that 
    \[\forall n\in \mathbb{N},\quad \|\alpha^n\|_T\leq C_b,\quad \|\bar{\alpha}^n\|_T\leq \max (C_b,\|\bar{\alpha}^0\|_T).\]
\end{lemma}
\begin{proof}
Let us fix $n\in \mathbb{N}$, by definition of $\alpha^n$ 
    \[\alpha^n=\text{arginf}J(\cdot,(\bar{m}^n_t)_{t\in [0,T]}).\]
    In particular using the strong convexity of $J(\cdot, (\bar{m}^n_t)_{t\in [0,T]})$, 
    \[\frac{c_L}{4}\|\alpha^n\|_T^2\leq J(0_{\left(\mathcal{H}^T\right)^d},(\bar{m}^n_t)_{t\in [0,T]})-J(\frac{\alpha^n}{2},(\bar{m}^n_t)_{t\in [0,T]}).\]
    By Hypothesis \ref{hyp: convexity in x}, $J$ is bounded from below by a constant $C_b$ independent of $(\bar{m}^n_t)_{t\in [0,T]}$. Moreover $J(0_{\left(\mathcal{H}^T\right)^d},(\bar{m}^n_t)_{t\in [0,T]})$ is also bounded independently of $(\bar{m}^n_t)_{t\in [0,T]}$. Consequently, there exists a constant $C'_b$ depending only on $G,L$ such that 
    \[\forall n\in \mathbb{N},\quad \|\alpha^n\|_T\leq C'_b.\]
    By recursion, for any $n\in\mathbb{N}$, 
    \[\|\bar{\alpha}^n\|_T\leq \max (C'_b,\|\bar{\alpha}^0\|_T).\]
\end{proof}
\begin{remarque}
To obtain this result, a version of the assumption \eqref{eq: boundedness of coef in measure at 0}, ensuring that $L(0,0,m),g(0,\mu)$ are bounded independently of the measure argument is not necessary. Typically if we assume instead that $x\mapsto \nabla_x G(x,\mu), (x,a)\mapsto (\nabla L)(x,a,m)$ have linear growth uniformly in $(\mu,m)$ then a uniform bound on the sequence of controls $(\alpha^n)_{n\in \mathbb{N}}$ can also be obtained. 
\end{remarque}
We now turn to an estimate on $(\|\alpha^{n+1}-\alpha^n\|^2_T)_{n\in \mathbb{N}}$:
\begin{lemma}
    \label{lemma: bound increment on control, fictitious play}
    Under Hypotheses \ref{hyp: regularite coefficient mfg}, \ref{hyp: convexity in x}, there exists a constant $C$ depending only on the Lipschitz constants of $\nabla_x G,\nabla_x L,\nabla_\alpha L$ in the measure argument and $T$ such that
    \[\forall n\in \mathbb{N}, \quad \|\alpha^{n+1}-\alpha^n\|_T\leq C\eta_n\|\alpha^n-\bar{\alpha}^n\|_T\]
\end{lemma}
\begin{proof}
    Fixing $n\in \mathbb{N}$, by the strong convexity of $J$ we deduce that 
    \[\frac{c_L}{4}\|\alpha^{n+1}-\alpha^n\|^2_T\leq \frac{1}{2}\left(J(\alpha^{n+1},(\bar{m}^n_t)_{t\in [0,T]})+J(\alpha^n,(\bar{m}^n_t)_{t\in [0,T]})\right)-J\left(\frac{\alpha^n+\alpha^{n+1}}{2},(\bar{m}^n_t)_{t\in [0,T]}\right).\]
    By definition of $\alpha^n$ as a minimizer we get 
    \[\frac{c_L}{4}\|\alpha^{n+1}-\alpha^n\|^2_T\leq \frac{1}{2}\left(J(\alpha^{n+1},(\bar{m}^n_t)_{t\in [0,T]})-J\left(\frac{\alpha^n+\alpha^{n+1}}{2},(\bar{m}^n_t)_{t\in [0,T]}\right)\right).\]
    We now remark that for two controls $\alpha,\beta$ such that
    \[\|\alpha\|_T,\|\beta\|_T<+\infty ,\]
    there exists a constant $C_{Lip}$ depending only on the Lipschitz constants of $\nabla_xg,\nabla_xL,\nabla_\alpha L$ in the measure argument such that for any two flow of measures  $(m^1_t)_{t\in[0,T]},(m^2_t)_{t\in [0,T]} \subset \mathcal{P}_2(\reels^{2d})$
    \begin{gather*}\left|(J(\alpha,(m^1_t)_{t\in[0,T]})-J(\beta,(m^1_t)_{t\in[0,T]}))-(J(\alpha,(m^2_t)_{t\in[0,T]})-J(\beta,(m^2_t)_{t\in[0,T]}))\right|\\\leq C_{Lip}(1+T)\|\alpha-\beta\|_T\sqrt{\esp{\left(\mathcal{W}^2_2(m^1_T,m^2_T)+\int_0^T \mathcal{W}^2_2(m^1_t,m^2_t)dt\right)}}.\end{gather*}
    In consequence 
    \begin{align*}\frac{c_L}{4}\|\alpha^{n+1}-\alpha^n\|^2_T\leq& \frac{1}{2}\left(J(\alpha^{n+1},(\bar{m}^{n+1}_t)_{t\in [0,T]})-J\left(\frac{\alpha^n+\alpha^{n+1}}{2},(\bar{m}^{n+1}_t)_{t\in [0,T]}\right)\right)\\
    &+\frac{1}{2}C_{Lip}(1+T)^2\|\alpha^{n+1}-\alpha^n\|_T\left(\|\bar{\alpha}^{n+1}_T-\bar{\alpha}^n_T\|+\|\bar{\alpha}^{n+1}-\bar{\alpha}^n\|_T\right)
    \end{align*}
    Finally using the definition of $\alpha^{n+1}$ as a minimizer we get that there exists a constant $C$ depending only on the Lipschitz constants of $\nabla_x G,\nabla_x L,\nabla_\alpha L$ in the measure argument and $T$ such that
    \[\forall n\in \mathbb{N}, \quad \|\alpha^{n+1}-\alpha^n\|_T\leq C\eta_n\|\alpha^n-\bar{\alpha}^n\|_T\]
\end{proof}
Before turning to the convergence of the fictitious play procedure, we need one final technical lemma 
\begin{lemma}
    \label{lemma: step size sequence}
    Let $(a_n)_{n\in \mathbb{N}}\subset \reels$ and $(\eta_n)_{n\in\mathbb{N}}\subset [0,1]$ be a sequence of decreasing step size satisfying 
    \[\sum_{n\in \mathbb{N}}\eta_n=+\infty, \quad \sum_{n\in \mathbb{N}}\eta_n^2<+\infty.\]
    If 
    \[\left|\sum_{n\in \mathbb{N}} \eta_n a_n\right|<+\infty, \]
    and 
    \[\exists C>0,\quad \forall n\in \mathbb{N} \quad |a_{n+1}-a_n|\leq C\eta_n,\]
    then 
    \[\lim_{n\to \infty} a_n=0.\]
\end{lemma}
\begin{proof}
We first introduce the sequence of time steps $(t_n)_{n\in\mathbb{N}}$ defined by 
\[t_0=0,\quad \forall n\in \mathbb{N} \quad t_{n+1}=\sum_{k=0}^n \eta_k.\]
By assumption on $(\eta_n)_{n\in \mathbb{N}}$, 
\[\lim_{n\to \infty}t_n=+\infty.\]
Thus, the function $f$ defined by 
\[f(0)=a_0, \quad \forall n\in \mathbb{N}, \forall t\in [t_{n+1},t_n], \quad f(t)=\frac{a_{n+1}-a_n}{\eta_n}(t-t_n)+a_n,\]
is continuous and well defined on $[0,+\infty)$. Moreover 
\begin{align*}\int_{t_n}^{t_{n+1}} f(t)dt&=\eta_n a_n+\frac{1}{2}(a_{n+1}-a_n)\eta_n.\\
    &\leq \eta_n a_n+\frac{1}{2}C\eta_n^2.
\end{align*}
Naturally 
\[\left|\int_0^{+\infty}f(t)dt\right|=\left|\sum_{n\in \mathbb{N}}\int_{t_n}^{t_{n+1}} f(t)dt \right|\leq \left|\sum_{n\in \mathbb{N}}\eta_n a_n\right|+\frac{1}{2}C\sum_{n\in \mathbb{N}}\eta_n^2<+\infty.\]
Since $f$ is Lipschitz with constant $C$, it is uniformly continuous, as its Riemann integral on $[0,+\infty)$ converges to a finite limit, by Barbalat's Lemma
\[\lim_{t\to \infty }f(t)=0.\]
In particular 
\[\lim_{n\to \infty} f(t_{n+1})=\lim_{n\to \infty}a_n=0.\]
\end{proof}
We now have all the tools at hand to conclude on the convergence of the sequence $(\|\alpha^n-\bar{\alpha}^n\|_T)_{n\in \mathbb{N}}$, 
\begin{corol}
    \label{corol: difference fictitious play tends to 0}
  Under Hypotheses \ref{hyp: potential mfg}, \ref{hyp: regularite coefficient mfg} and \ref{hyp: convexity in x}, letting the sequence $(\alpha^n,\bar{\alpha}^n)_{n\in \mathbb{N}}$ be defined as in \eqref{eq: fictitious play}, if 
  \[\sum_{n\in \mathbb{N}}\eta_n=+\infty, \text{ and } \sum_{n\in \mathbb{N}}\eta_n^2<+\infty,\]
  then 
  \[\lim_{n\to \infty} \|\alpha^n-\bar{\alpha}^n\|_T=0.\]
\end{corol}
\begin{proof}
    First by Lemma \ref{lemma: boundedness in H^T, fictitious play}, the sequence $(\alpha^n,\bar{\alpha}^n)_{n\in \mathbb{N}}$ is bounded in $\left(\mathcal{H}^T\right)^{2d}$. Moreover by Lemma \ref{lemma: decreasing potential fictitious play}, for any $n\in \mathbb{N}$
    \[\mathcal{J}(\bar{\alpha}^{n+1})-\mathcal{J}(\bar{\alpha}^n)\leq \left(-c_L \eta_n+C_{Lip}(1+T)\eta_n^2\right)\|\alpha^n-\bar{\alpha}^n\|^2_T.\]
    Summing thoses inequalities over $n\in \mathbb{N}$ yields that for any $n\in \mathbb{N}$
    \[c_L \sum_{k=0}^n \eta_k\|\alpha^k-\bar{\alpha}^k\|^2_T\leq \mathcal{J}(\bar{\alpha}^0)-\mathcal{J}(\bar{\alpha}^{n+1})+C\sum_{k=0}^n \eta_k^2,\]
    for some constant $C$ independent of $n\in \mathbb{N}$, as we have used the fact that the sequence $(\|\alpha^n-\bar{\alpha}^n\|_T^2)_{n\in \mathbb{N}}$ is uniformly bounded in $\reels$. By assumption $l,g$ are bounded from below and consequently so is $\mathcal{J}$, moreover 
    \[\sum_{n\in \mathbb{N}}\eta_n^2<+\infty,\]
    and we deduce that 
    \[\sum_{n\in \mathbb{N}}\eta_n \|\alpha^n-\bar{\alpha}^n\|^2_T<+\infty.\]
    We now observe that for fixed $n\in \mathbb{N}$
    \begin{align*}
        \|\alpha^{n+1}-\bar{\alpha}^{n+1}\|^2_T-\|\alpha^{n}-\bar{\alpha}^{n}\|_T&=\langle \alpha^{n+1}-\bar{\alpha}^{n+1}+\alpha^{n}-\bar{\alpha}^{n}, \alpha^{n+1}-\bar{\alpha}^{n+1}-(\alpha^{n}-\bar{\alpha}^{n})\rangle^T\\
        &\leq \left(\|\alpha^{n+1}-\alpha^{n}\|_T+\|\bar{\alpha}^{n+1}-\bar{\alpha}^n\|_T\right)\|\alpha^{n+1}-\bar{\alpha}^{n+1}-(\alpha^{n}-\bar{\alpha}^{n})\|_T\\
        &\leq C\eta_n,
    \end{align*}
    where the last line holds from Lemma \ref{lemma: bound increment on control, fictitious play} and the boundednes of $(\alpha^n,\bar{\alpha}^n)_{n\in \mathbb{N}}$ in $\left(\mathcal{H}^T\right)^{2d}$. We may finally apply Lemma \ref{lemma: step size sequence} with 
    \[\forall n\in \mathbb{N}, \quad a_n=\|\alpha^n-\bar{\alpha}^n\|^2_T, \]
    ending the proof. 
\end{proof}
\begin{remarque}
Let us remark at this point that if we change assumption \ref{hyp: convexity in x} to the stronger assumption that the Lagrangian $L$ is strongly convex in $(x,\alpha)$ instead of convex strongly in $\alpha$, and that $g$ is also strongly convex then we can get a bound on 
\[\esp{ \left|X^{\alpha^n}_T-X^{\bar{\alpha}^n}_T\right|^2 +\int_0^T \left|X^{\alpha^n}_t-X^{\bar{\alpha}^n}_t\right|^2dt+\int_0^T \left|\alpha^n_t-\bar{\alpha}^n_t\right|^2dt},\]
from the strong convexity of coefficients. In particular in this case we do not need to use the estimate 
\[\forall t\in [0,T],\quad \esp{ \left|X^{\alpha^n}_t-X^{\bar{\alpha}^n}_t\right|^2dt}\leq t\esp{\int_0^t\left|\alpha^n_s-\bar{\alpha}^n_s\right|^2ds},\]
and all bounds we obtain are independent of the time horizon $T$ meaning it then becomes possible to study an ergodic version of the fictitious play.
\end{remarque}
\subsubsection{Convergence of the fictitious through relaxed controls}
The main difficulty from the addition of common noise is the lack of compactness it causes on the sequence $(\bar{m}^n_t)_{t\in [0,T]}$ which is a measure valued stochastic process. To deal with this issue we weaken the notion of equilibrium considered and use the notion of measure valued MFG equilibrium introduced in \cite{djete-measure-valued}. Borrowing notations from this very article, we denote by $\mathbb{M}$, the set of all probability measures in $\mathcal{P}([0,T]\times \mathcal{P}(\reels^{2d}))$ whose marginal distribution on $[0,T]$ is given by Lesbegue measure. We define
\[\bar{\Omega}=\mathcal{C}([0,T],\mathcal{P}(\reels^d))\times\mathbb{M}\times \mathcal{C}([0,T],\reels^d),\]
$(\mu,\Lambda,B^0)$ to be the canonical element on $\bar{\Omega}$ and $\left(\bar{\mathcal{F}}_t\right)_{t\in [0,T]}$ the associated canonical filtration, as defined in \cite{djete-measure-valued} section 2.3.1. We do not insist too much on measurability concerns which are already treated in the same section of Djete's article. In particular there exists a disintegration of $\Lambda$ of the form 
\[\Lambda(dm,dt)=\Lambda_t(dm)dt,\]
with $(\Lambda_t)_{t\in [0,T]}$ a $\bar{\mathcal{F}}-$predictable process. For a function $f\in \mathcal{C}^2_b(\reels^d,\reels)$, we denote by 
\begin{align*}N_t(f)&=\int_{\reels^d}f(x-\sigma^0 B^0_t)\mu_t(dx)-\int_{\reels^d}f(x)\mu_0(dx)\\
&-\int_0^t\int_{\mathcal{P}(\reels^{2d})} \left(\nabla_x f(x-B^0_s)\cdot \alpha +\frac{1}{2}\sigma^2 \Delta_x f(x-B^0_s)\right)m(dx,d\alpha)\Lambda (ds,dm),
\end{align*}
Let us remark that $N_t(f)$ is not finite in general for elements of $\mathcal{P}([0,T]\times \mathcal{P}(\reels^{2d}))$ due to the unbounded drift, however this is not a problem if $\bar{\Omega}$ is endowed with a probability measure $\mathbb{Q}$ such that \[\int_{[0,T]\times \mathcal{P}(\reels^{2d})}\int_{\reels^{2d}}|(x,a)|m(dx,da)\Lambda(ds,dm)<+\infty,\quad \mathbb{Q}-a.s. \]
Finally for any measure $\mu\in \mathcal{P}(\reels^d)$, we let 
\[\mathbb{Z}_\mu=\left\{m\in \mathcal{P}(\reels^{2d}), \quad \int_{a\in \reels^d}m(dx,da)=\mu(dx)\right\}.\]
\begin{definition}
    \label{def: measure valued rule}
Given an initial distribution $m_0$ for players, a probability $\mathbb{Q}\in \mathcal{P}(\bar{\Omega})$, is a measure valued rule if
\begin{enumerate}
    \item[-]$(\mathbb{Q}(\mu_0=m_0))=1$.
    \item[-] $(B^0_t)_{t\in [0,T]}$ is a $(\mathbb{Q},(\bar{\mathcal{F}}_t)_{t\geq 0})$ Brownian motion, \[\int_{[0,T]\times \mathcal{P}(\reels^{2d})}\int_{\reels^{2d}}|(x,a)|m(dx,da)\Lambda(ds,dm)<+\infty,\quad \mathbb{Q}-a.s, \]and for any $f\in \mathcal{C}^2_b(\reels^d,\reels)$, 
    \[N_t(f)=0 \quad \mathbb{Q}-\text{a.s.}\]
    \item[-] There exists a measurable subset $N\subset [0,T]\times \bar{\Omega}$ such that 
    \[\int_{(t,\omega)\in N} \mathbb{Q}(d\omega)dt=0,\]
    and 
    \[\Lambda_t(\mathbb{Z}_{\mu_t})=1, \quad \forall (t,\omega)\in [0,T]\times \bar{\Omega}\setminus N\]
\end{enumerate}
\end{definition}
This definition may seem quite heavy, let us insist that this is just a way of restraining ourselves to the set of probability measures $\mathbb{Q}$ which are natural in our setting: The first condition means the inital distribution of players must be given by $\mu_0$, the second and third conditions are just a way to assert that the state equation must be satisfied in some weak sense. It is easy to check that for any admissible control $\alpha$ those conditions are verified with 
\begin{equation}
    \label{eq: from cost to lambda}
    \Lambda(dm,dt)=\delta_{\mathcal{L}(X^\alpha_t,\alpha_t|\mathcal{F}^0_t)}(dm)dt,
\end{equation}
under $\mathbb{P}$. For readers who are not familiar with the notion of relaxed controls, let us remark that for a given control, the cost function can be expressed in term of this $\Lambda$ by 
\[\esp{\int_0^T L(t,X^\alpha_t,\alpha_t,\mathcal{L}(X^\alpha_t,\alpha_t))}=\esp{\int_{[0,T]\times \mathcal{P}(\reels^{2d})}\int_{\reels^{2d}}L(t,x,\alpha,m)m(dx,d\alpha)\Lambda(dt,dm)}.\]
The choice of looking directly at this $\Lambda$ insteads of controls, is because it is much easier to get compactness on a sequence of measures. On the other hand, people who are familiar with relaxed controls but not in the context of mean field games may be surprised that there is still an expectation on the rightside of \eqref{eq: from cost to lambda}. Unfortunately, due to the presence of common noise $\Lambda$ is in general stochastic, this is why formally we work at the level of $\mathcal{L}(\Lambda)$, and the definition of an admissible measure rule Definition \ref{def: measure valued rule} is so convoluted.
\begin{definition}
    \label{def: measure valued mfg}
    A measure valued rule $\mathbb{Q}\in \mathcal{P}(\bar{\Omega})$ is said to be a measure valued MFG equilibrium if for any bounded, progressively measurable, uniformly Lipschitz in $x$ function $\alpha:[0,T]\times \reels^d\times \mathbb{M}\to \reels^d$, letting $(\mu^\alpha_t)_{t\in [0,T]}$ be the solution to 
    \[d\mu^\alpha_t=-\text{div}\left(\left(\alpha(t,x,\Lambda)dt+\sigma^0dW^0_t\right)\mu^\alpha_t\right)+\frac{1}{2}(\sigma^2+(\sigma^0)^2) \Delta_x \mu^\alpha_tdt, \quad \mu^\alpha_0=\mu_0 \quad \mathbb{Q}-\text{a.s,}\]
    the following holds 
    \begin{gather*}\mathbb{E}^\mathbb{Q}\left[\int_{\reels^d} G(x,\mu_T)\mu_T(dx)+\int_0^T\int_{m\in \mathbb{M}}\int_{(x,a)\in \reels^{2d}}L(t,x,a,m)m(dx,d\alpha)\Lambda(dt,dm)\right]\\
    \leq \mathbb{E}^\mathbb{Q}\left[\int_{\reels^d} G(x,\mu_T)\mu^\alpha_T(dx)+\int_0^T\int_{m\in \mathbb{M}}\int_{x\in\reels^d}L(t,x,\alpha(t,x,\Lambda),m)\mu^\alpha_t(dx)\Lambda(dt,dm)\right]
    \end{gather*}
\end{definition}
\begin{remarque}
    The fact that we only consider bounded and Lipschitz feedback controls as admissible is mainly a choice of convenience. This will make proofs easier and is non consequential as less regular controls may be approached by controls in this class. Typically it is straightforward to prove that the infimum reached is the same whether or not we require controls to be bounded. Indeed an unbounded control $\alpha$ can be approached by a sequence of bounded controls $\alpha^n=\alpha \wedge (-n)\vee n$ and the conclusion follows from Fatou's Lemma. Removing the Lipschitz assumption is possible but slightly more technical as in this case, there might be several solutions to the Fokker-Planck equation
    \[d\mu^\alpha_t=-\text{div}\left(\left(\alpha(t,x,\Lambda)dt+\sigma^0dW^0_t\right)\mu^\alpha_t\right)+\frac{1}{2}(\sigma^2+(\sigma^0)^2) \Delta_x \mu^\alpha_tdt, \quad \mu^\alpha_0=\mu_0 \quad \mathbb{Q}-\text{a.s,}.\]
\end{remarque}
\begin{thm}
    \label{thm: convergence to measure valued mfg}
     Under Hypotheses \ref{hyp: potential mfg}, \ref{hyp: regularite coefficient mfg} and \ref{hyp: convexity in x}, letting $(\alpha^n,\bar{\alpha}^n)_{n\in \mathbb{N}}$ be the sequence defined through the fictitious play procedure and considering
    \[\bar{\mu}^n_t=\mathcal{L}(X^{\bar{\alpha}^n}_t|\mathcal{F}^0_t),\]
    and
    \[\Lambda^n(dt,dm)=\delta_{\mathcal{L}(X^{\bar{\alpha}^n}_t,\bar{\alpha}^n_t|\mathcal{F}^0_t)}(dm)dt,\]
    the sequence $\mathbb{Q}^n=\mathcal{L}\left((\bar{\mu}^n_t)_{t\in [0,T]},\Lambda^n,(W^0_t)_{t\in [0,T]}\right)$ is precompact in $\mathcal{P}_q(\bar{\Omega})$ for the metric $\mathcal{W}_q$ for any $1\leq q<2$. Moreover, if there exists $\eta>0$ such that $\mu_0=\mathcal{L}(X_0)\in \mathcal{P}_{2+\eta}(\reels^d)$, then any limit point $\mathbb{Q}^*$ of $(\mathbb{Q}^n)_{n\in \mathbb{N}}$ is a measure valued MFG equilibrium.
\end{thm}
\begin{proof}
    \quad

    \flushleft\textit{First step: any limit point of $(\mathbb{Q}^n)_{n\in \mathbb{N}}$ is a measure valued rule}\\
    It is quite straightforward that for any given $n\in \mathbb{N}$ $\mathbb{Q}^n$ is a measure valued rule. The compacteness of the sequence follows from Proposition A.2 and B.1 of \cite{mfg-with-common}. Let us assume that along a subsequence, $(\mathbb{Q}^n)_{n\in \mathbb{N}}$ converges to some limit $\mathbb{Q}$. By an abuse of notation this subsequence is also denoted $(\mathbb{Q}^n)_{n\in \mathbb{N}}$. Let us recall that for fixed $n\in \mathbb{N}$, $(\bar{\mu}^n_t)_{t\in [0,T]}$ is given as the conditional law with respect to the common noise of the random variable $(\bar{X}^n_t)_{t\in [0,T]}$ with 
    \[\bar{X}^n_t=X_0-\int_0^t \bar{\alpha}^n_s ds+\sigma W_t+\sigma^0 W^0_t.\] 
    It follows that for any function $f\in C^2_b(\reels^d,\reels)$, 
    \begin{align*}N^n_t(f)&=\espcond{f(\bar{X}^n_t-\sigma^0 W^0_t)-f(X_0)-\int_0^t \left(\nabla_x f(\bar{X}_s-\sigma^0 W^0_s)\cdot \bar{\alpha}^n_sds+\frac{1}{2}\sigma^2 \Delta_x f(\bar{X}_s-\sigma^0 W^0_s)\right)ds}{\mathcal{F}_t}\\
    &=0 \quad \mathbb{P}-a.s.
    \end{align*}
    In particular 
    \[\forall n\in \mathbb{N}, \quad \mathbb{E}^\mathbb{P}\left[|N^n_t(f)|\right]=0.\]
    On the other hand 
    \[\mathbb{E}^\mathbb{P}\left[|N^n_t(f)|\right]=\mathbb{E}^{\mathbb{Q}^n}\left[|N_t(f)|\right].\]
    By the convergence of $(\mathbb{Q}^n)_{n\in \mathbb{N}}$ to $\mathbb{Q}$ for $\mathcal{W}_1$ we deduce immediately that 
    \[\mathbb{E}^{\mathbb{Q}}\left[|N_t(f)|\right]=0,\]
    and hence that for any $f\in C^2_b(\reels^d,\reels)$,
    \[N_t(f)=0,\quad \mathbb{Q}-a.s.\]
    Next we prove that 
    \[\Lambda_t(\mathbb{Z}_{\mu_t})=1,\quad \lambda \otimes\mathbb{Q} \text{ on }[0,T]\times \bar{\Omega},\]
    where $\lambda$ indicates Lesbegues measure. To that end let us fix a function $h\in C_b(\reels^d,\reels)$, for fixed $n\in \mathbb{N}$,
    \begin{align*}0&=\esp{\int_0^T\int_{m\in \mathcal{P}(\reels^{2d})}\left|\int_{\reels^{2d}}h(x)(\bar{\mu}^n_t(dx)-m(dx,da))\right|\Lambda^n(dt,dm)}\\
    &=\mathbb{E}^{\mathbb{Q}^n}\left[\int_0^T\int_{m\in \mathcal{P}(\reels^{2d})}\left|\int_{\reels^{2d}}h(x)(\mu_t(dx)-m(dx,da))\right|\Lambda(dt,dm)\right].\end{align*}
    Moreover by \cite{djete-measure-valued}, proposition A.4, the mapping 
    \[(\Lambda,(\mu_t)_{t\in [0,T]})\mapsto\int_0^T\int_{m\in \mathcal{P}(\reels^{2d})}\left|\int_{\reels^{2d}}h(x)(\mu_t(dx)-m(dx,da))\right|\Lambda(dt,dm), \]
    is continuous, and it follows by passing to the limit that 
    \[\mathbb{E}^{\mathbb{Q}}\left[\int_0^T\int_{m\in \mathcal{P}(\reels^{2d})}\left|\int_{\reels^{2d}}h(x)(\mu_t(dx)-m(dx,da))\right|\Lambda(dt,dm)\right]=0.\]
    As this is true for any continuous bounded function $h$, we deduce that 
    \[\Lambda_t(\mathbb{Z}_{\mu_t})=1, \lambda \otimes \mathbb{Q}-a.s.\]
    The other properties of measure valued rules are verified uniformly in $n\in \mathbb{N}$ and there is no particular difficulty in proving that they are conserved by the limiting measure $\mathbb{Q}$.
\flushleft\textit{Second step: any limit point of $(\mathbb{Q}^n)_{n\in \mathbb{N}}$ is a measure valued MFG equilibrium}\\
Let us fix $n\in \mathbb{N}$, by definition of the couple $(X^n_t,\alpha^n_t)_{t\in [0,T]}$ for any admissible control $(\alpha_t)_{t\in [0,T]}$ the following holds
\begin{gather*}
\esp{G(X^n_T,\mathcal{L}(\bar{X}^n_T|\mathcal{F}^0_T))+\int_0^T L(t,X^n_t,\alpha^n_t,\mathcal{L}(\bar{X}^n_t,\bar{\alpha}^n_t|\mathcal{F}^0_t))dt}\\
\leq \esp{G(X^\alpha_T,\mathcal{L}(\bar{X}^n_T|\mathcal{F}^0_T))+\int_0^T L(t,X^\alpha_t,\alpha_t,\mathcal{L}(\bar{X}^n_t,\bar{\alpha}^n_t|\mathcal{F}^0_t))dt}.
\end{gather*}
Using the growth assumptions \eqref{eq: linear growth of coefs} of Hypothesis \ref{hyp: regularite coefficient mfg} and Cauchy-Schwartz inequality, we deduce that there exists a constant $C>0$ such that for any admissible control $\alpha$
\begin{gather*}
\esp{G(\bar{X}^n_T,\mathcal{L}(\bar{X}^n_T|\mathcal{F}^0_T))+\int_0^T L(t,\bar{X}^n_t,\bar{\alpha}^n_t,\mathcal{L}(\bar{X}^n_t,\bar{\alpha}^n_t|\mathcal{F}^0_t))dt}\\
=\mathbb{E}^{\mathbb{Q}^n}\left[\int_{\reels^d} G(x,\mu_T)\mu_T(dx)+\int_0^T\int_{m\in \mathbb{M}}\int_{(x,a)\in \reels^{2d}}L(t,x,a,m)m(dx,d\alpha)\Lambda(dt,dm)\right]\\
\leq \esp{G(X^\alpha_T,\mathcal{L}(\bar{X}^n_T|\mathcal{F}^0_T))+\int_0^T L(t,X^\alpha_t,\alpha_t,\mathcal{L}(\bar{X}^n_t,\bar{\alpha}^n_t|\mathcal{F}^0_t))dt}\\
+C\left(1+\|\alpha^n\|_T+\|\bar{\alpha}^n\|_T\right)\|\alpha^n-\bar{\alpha}^n\|_T.
\end{gather*}
Letting $a:[0,T]\times \reels^d\times \mathbb{M}\to \reels^d$ be a Lipschitz in $x$, bounded and progressively measurable function we consider the unique solution $(X^a_t)_{t\in [0,T]}$ of 
\[X^a_t=X_0-\int_0^t a(t,X^a_t,\Lambda^n)dt+\sigma W_t+\sigma^0W^0_t,\]
and the associated control 
\[\forall t\in [0,T],\quad \alpha_t=a(t,X^a_t,\Lambda^n).\]
Since $a$ is Lipschitz in the space variable, $\mu^a_t=\mathcal{L}(X^a_t|\mathcal{F}^0_t)$, is the unique solution of the stochastic Fokker-Planck equation
\[d\mu^{n,a}_t=-\text{div}\left(\left(a(t,x,\Lambda^n)dt+\sigma^0dW^0_t\right)\mu^{n,a}_t\right)+\frac{1}{2}(\sigma^2+(\sigma^0)^2) \Delta_x \mu^{n,a}_tdt, \quad \mu^{n,a}_0=\mu_0 \quad \mathbb{P}-\text{a.s,}.\]
In particular, since $a$ is a bounded function and by assumption 
\[\esp{|X_0|^{2+\eta}}<+\infty,\]
there exists a constant independent of $n\in \mathbb{N}$, $C^a$, such that 
\begin{equation}
    \label{eq: boundedness in expectation of mu a}
    \forall n\in \mathbb{N},\forall t\in [0,T]\quad \esp{\int_{\reels^d}|x|^{2+\eta}\mu^{n,a}_t(dx)}<C^a. \end{equation}
Moreover, for any $n\in \mathbb{N}$
\begin{gather*}
\esp{G(X^\alpha_T,\mathcal{L}(\bar{X}^n_T|\mathcal{F}^0_T))+\int_0^T L(t,X^\alpha_t,\alpha_t,\mathcal{L}(\bar{X}^n_t,\bar{\alpha}^n_t|\mathcal{F}^0_t))dt}\\
=\mathbb{E}\left[\int_{\reels^d} G(x,\mu^n_T)\mu^{n,a}_t(dx)+\int_0^T\int_{m\in \mathbb{M}}\int_{x\in\reels^d}L(t,x,a(t,x,\Lambda^n),m)\mu^{n,a}_t(dx)\Lambda^n(dt,dm)\right]\\
=\mathbb{E}^{\mathbb{Q}^n}\left[\int_{\reels^d} G(x,\mu_T)\mu^a_T(dx)+\int_0^T\int_{m\in \mathbb{M}}\int_{x\in\reels^d}L(t,x,a(t,x,\Lambda),m)\mu^a_t(dx)\Lambda(dt,dm)\right],
\end{gather*}
where $(\mu^a_t)_{t\in [0,T]}$ solves 
\[d\mu^{a}_t=-\text{div}\left(\left(a(t,x,\Lambda)dt+\sigma^0dB^0_t\right)\mu^{a}_t\right)+\frac{1}{2}(\sigma^2+(\sigma^0)^2) \Delta_x \mu^{a}_tdt, \quad \mu^{a}_0=\mu_0,\quad \mathbb{Q}^n-a.s.\]
It follows that there exists a constant $C'$ independent of the choice of $a$ such that 
\begin{gather*}
\mathbb{E}^{\mathbb{Q}^n}\left[\int_{\reels^d} G(x,\mu_T)\mu_T(dx)+\int_0^T\int_{m\in \mathbb{M}}\int_{(x,a)\in \reels^{2d}}L(t,x,a,m)m(dx,d\alpha)\Lambda(dt,dm)\right]\\
\leq \mathbb{E}^{\mathbb{Q}^n}\left[\int_{\reels^d} G(x,\mu_T)\mu^a_T(dx)+\int_0^T\int_{m\in \mathbb{M}}\int_{x\in\reels^d}L(t,x,a(t,x,\Lambda),m)\mu^a_t(dx)\Lambda(dt,dm)\right]\\
+C'\left(1+\|\alpha^n\|_T+\|a\|_\infty\right)\|\alpha^n-\bar{\alpha}^n\|_T
\end{gather*}
Let us now assume that (up to the extraction of a subsequence), $\mathbb{Q}^n$ converges to some limit $\mathbb{Q}$. By Hypothesis \ref{hyp: convexity in x}, $G,L$ are bounded from below and it follows from Fatou's Lemma that 
\begin{gather*}
    \mathbb{E}^{\mathbb{Q}}\left[\int_{\reels^d} G(x,\mu_T)\mu_T(dx)+\int_0^T\int_{m\in \mathbb{M}}\int_{(x,\alpha)\in \reels^{2d}}L(t,x,a,m)m(dx,d\alpha)\Lambda(dt,dm)\right]\\
    \leq \liminf_n \mathbb{E}^{\mathbb{Q}^n}\left[\int_{\reels^d} G(x,\mu_T)\mu_T(dx)+\int_0^T\int_{m\in \mathbb{M}}\int_{(x,\alpha)\in \reels^{2d}}L(t,x,a,m)m(dx,d\alpha)\Lambda(dt,dm).\right]
\end{gather*}
On the other hand, letting $\varepsilon\in [0,\frac{\eta}{2+\eta}\wedge \frac{1}{2})$, we observe that there exists a constant $C$ such that for any $n\in \mathbb{N}$
\begin{gather*}
    \mathbb{E}^{\mathbb{Q}^n}\left[\left|\int_{\reels^d} G(x,\mu_T)\mu^a_T(dx)+\int_0^T\int_{m\in \mathbb{M}}\int_{x\in\reels^d}L(t,x,a(t,x,\Lambda),m)\mu^a_t(dx)\Lambda(dt,dm)\right|^{1+\varepsilon}\right]\\
    \leq C+C\mathbb{E}^{\mathbb{Q}^n}\left[\mathcal{W}_2(\mu_T,\delta_{0_{\reels^d}})^{2}+\int_{\reels^d} |x|^{2+\eta}\mu^a_T(dx)\right]\\
    +C\mathbb{E}^{\mathbb{Q}^n}\left[\int_0^T\int_{m\in \mathbb{M}}\int_{x\in\reels^d}\left(|x|^{2+\eta}+|a(t,x,\Lambda)|^{2+\eta}+\int_{(y,\beta)\in \reels^{2d}}|(y,\beta)|^2m(dy,d\beta)\right)\mu^a_t(dx)\Lambda(dt,dm)\right],
\end{gather*}
where we have used our growth assumptions on coefficients Hypothesis \ref{hyp: regularite coefficient mfg} and Hypothesis \ref{hyp: convexity in x}, as well as Young's inequality. We now recall that 
\begin{gather*}\mathbb{E}^{\mathbb{Q}^n}\left[\int_0^T\int_{m\in \mathbb{M}}\int_{(y,\beta)\in \reels^{2d}}|(y,\beta)|^2m(dy,d\beta)\Lambda(dt,dm)\right]\\
=\esp{\int_0^T\int_{m\in \mathbb{M}}\int_{(y,\beta)\in \reels^{2d}}|(y,\beta)|^2m(dy,d\beta)\Lambda^n(dt,dm)}\\
=\esp{\int_0^T|(\bar{X}^n_t,\bar{\alpha}^n_t)|^2dt},
\end{gather*}
which is bounded uniformly in $n\in \mathbb{N}$ by Lemma \ref{lemma: boundedness in H^T, fictitious play}. We conclude by using \eqref{eq: boundedness in expectation of mu a} that 
\[\sup_n\mathbb{E}^{\mathbb{Q}^n}\left[\left|\int_{\reels^d} G(x,\mu_T)\mu^a_T(dx)+\int_0^T\int_{m\in \mathbb{M}}\int_{x\in\reels^d}L(t,x,a(t,x,\Lambda),m)\mu^a_t(dx)\Lambda(dt,dm)\right|^{1+\varepsilon}\right]<+\infty.\]
In particular it follows that the random variable 
\[\left|\int_{\reels^d} G(x,\mu_T)\mu^a_T(dx)+\int_0^T\int_{m\in \mathbb{M}}\int_{x\in\reels^d}L(t,x,a(t,x,\Lambda),m)\mu^a_t(dx)\Lambda(dt,dm)\right|,\]
is asymptotically uniform integrable with respect to the sequence of measure $(\mathbb{Q}^n)_{n\in \mathbb{N}}$, thus by a variation on Fatou's Lemma presented in \cite{fatou} Theorem 2.4
\begin{gather*}
\limsup_n \mathbb{E}^{\mathbb{Q}^n}\left[\int_{\reels^d} G(x,\mu_T)\mu^a_T(dx)+\int_0^T\int_{m\in \mathbb{M}}\int_{x\in\reels^d}L(t,x,a(t,x,\Lambda),m)\mu^a_t(dx)\Lambda(dt,dm)\right]\\
\leq \mathbb{E}^{\mathbb{Q}}\left[\int_{\reels^d} G(x,\mu_T)\mu^a_T(dx)+\int_0^T\int_{m\in \mathbb{M}}\int_{x\in\reels^d}L(t,x,a(t,x,\Lambda),m)\mu^a_t(dx)\Lambda(dt,dm)\right].
\end{gather*}
Finally by Lemma \ref{lemma: boundedness in H^T, fictitious play} and \ref{corol: difference fictitious play tends to 0} 
\[\lim_{n\to \infty}\left(1+\|\alpha^n\|_T+\|a\|_\infty\right)\|\alpha^n-\bar{\alpha}^n\|_T=0,\]
and we deduce that for any admissible function $a$
\begin{gather*}
    \mathbb{E}^{\mathbb{Q}}\left[\int_{\reels^d} G(x,\mu_T)\mu_T(dx)+\int_0^T\int_{m\in \mathbb{M}}\int_{(x,\alpha)\in \reels^{2d}}L(t,x,a,m)m(dx,d\alpha)\Lambda(dt,dm)\right]\\
    \leq \mathbb{E}^{\mathbb{Q}}\left[\int_{\reels^d} G(x,\mu_T)\mu^a_T(dx)+\int_0^T\int_{m\in \mathbb{M}}\int_{x\in\reels^d}L(t,x,a(t,x,\Lambda),m)\mu^a_t(dx)\Lambda(dt,dm)\right],
\end{gather*}
which proves that $\mathbb{Q}$ is a measure valued MFG equilibrium.
\end{proof}
\subsection{Strong convergence in the convex setting}
we now make the following convexity assumption on coefficients
\begin{hyp}
    \label{hyp: coupling convexity}
    There exists a constant $c_l>0$, such that for any $(X,\alpha),(Y,\beta)\in L^2(\Omega, \reels^d)$, $(t,\lambda)\in [0,T]\times [0,1]$, 
    \[g(\mathcal{L}(\lambda X+(1-\lambda)Y))\leq \lambda g(\mathcal{L}(X))+(1-\lambda)g(\mathcal{L}(Y))\]
    and 
    \[l(\mathcal{L}(t,\lambda(X,\alpha)+(1-\lambda)(Y,\beta)))\leq \lambda l (t,\mathcal{L}(X,\alpha))+(1-\lambda)l (t,\mathcal{L}(Y,\beta))-\frac{1}{2}c_l \lambda(1-\lambda)\|\alpha-\beta\|^2\]
\end{hyp}
We do not refer to this convexity as displacement convexity, which is usually reserved for convexity along optimal couplings, instead some other works suggest to use the terminology of coupling convexity \cite{bertucci-livre}. Let us insist however, that this notion of convexity implies that the coefficients $G=\frac{\delta g}{\delta \mu}$, $L=\frac{\delta l}{\delta m}$ are displacement monotone. In particular, it is a direct consequence of \cite{meszaros} Lemma 2.3, that under this assumption, Hypothesis \ref{hyp: convexity in x} is verified. 
\begin{lemma}
    Under Hypothesis \ref{hyp: regularite coefficient mfg} and \ref{hyp: coupling convexity} there exists a unique optimal control $\alpha^*\in \left(\mathcal{H}^T\right)^d$ to the mean field control problem 
    \begin{equation}\label{eq: mfc}\inf_\alpha \mathcal{J}(\alpha).\end{equation}
    Moreover, it is the unique control satisfying 
    \begin{gather}
   \nonumber\forall (\alpha'_t)_{t\in [0,T]}\in \left(\mathcal{H}^T\right)^d\\
    \label{eq: first order optimality mfc} \esp{\nabla_x G(X^{\alpha^*}_T,\mu^*_T)\cdot (X^{\alpha'}_T-X^{\alpha^*}_T)+\int_0^T \left(\begin{array}{c} \nabla_x L(t,X^{\alpha^*}_t,\alpha^*_t,m^*_t)\\ \nabla_\alpha L(t,X^{\alpha^*}_t,\alpha^*_t,m^*_t) \end{array}\right)\cdot \left(\begin{array}{c} X^{\alpha'}_t-X^{\alpha^*}_t\\ \alpha'_t-\alpha^*_t\end{array}\right) dt}\geq 0,\\
    \nonumber m^*_t=\mathcal{L}(X^{\alpha^*}_t,\alpha^*_t|\mathcal{F}^0_t),\\
    \nonumber \mu^*_T=\mathcal{L}(X^{\alpha^*}_T|\mathcal{F}^0_T).
\end{gather} 
\end{lemma}
\begin{proof}
The proof follow from the strong convexity of $\mathcal{J}$ on $\left(\mathcal{H}^T\right)^d$  under Hypothesis \ref{hyp: coupling convexity}, as well as arguments already presented in the proof of Lemma \ref{lemma: first order optimality mfg} using the first order optimality conditions at $\alpha^*$.
\end{proof}
Naturally, a direct consequence is that whenever the potential $\mathcal{J}$ is convex, there exists a unique MFG equilibrium and the sequence $(\alpha^n)_{n\in \mathbb{N}}$ converges strongly
\begin{thm}
    \label{thm: cv monotone setting}
Under Hypotheses \ref{hyp: regularite coefficient mfg} and \ref{hyp: coupling convexity}, if the sequence $(\eta_n)_{n\in \mathbb{N}}\subset [0,1]$ satisfies 
\[\sum_{n=1}^{+\infty} \eta_n=+\infty,\quad \sum_{n=1}^{+\infty}\eta_n^2<+\infty,\]

then, the sequence of controls $(\alpha^n)_{n\in\mathbb{N}}$ generated by the fictitious play converges strongly in $\left(\mathcal{H}^T\right)^d$ to the optimal control $\alpha^*$ solution of the mean field control problem \eqref{eq: mfc}.
\end{thm}
\begin{proof}
Let us insist once again that under Hypothesis \ref{hyp: coupling convexity}, Hypothesis \ref{hyp: convexity in x} is satisfied by $G,L$. Consequently the sequence $(\alpha^n,\bar{\alpha}^n)_{n\in \mathbb{N}}$ generated by the fictitious play is well defined, uniformly bounded in $\left(\mathcal{H}^T\right)^d$ by Lemma \ref{lemma: boundedness in H^T, fictitious play} and such that 
\[\lim_{n\to \infty} \| \alpha^n-\bar{\alpha}^n\|_T=0,\]
by Lemma \ref{corol: difference fictitious play tends to 0}. On the other hand, using the first order optimality conditions Lemma \ref{lemma: first order optimality mfg}, for any $n\in \mathbb{N}$
\begin{gather*}
\forall (\alpha_t)_{t\in [0,T]}\in \left(\mathcal{H}^T\right)^d\\
     \esp{\nabla_x G(X^{\alpha^n}_T,\bar{\mu}^n_T)\cdot (X^{\alpha}_T-X^{\alpha^n}_T)+\int_0^T \left(\begin{array}{c} \nabla_x L(t,X^{\alpha^n}_t,\alpha^n_t,\bar{m}^n_t)\\ \nabla_\alpha L(t,X^{\alpha^n}_t,\alpha^n_t,\bar{m}^n_t) \end{array}\right)\cdot \left(\begin{array}{c} X^{\alpha}_t-X^{\alpha^n}_t\\ \alpha_t-\alpha^n_t\end{array}\right) dt}\geq 0,\\
\bar{\mu}^n_T=\mathcal{L}(X^{\bar{\alpha}^n}_T|\mathcal{F}^0_T),\\ \forall t\in [0,T],\quad \bar{m}^n_t=\mathcal{L}(X^{\bar{\alpha}^n}_t,\bar{\alpha}^n_t|\mathcal{F}^0_t).
\end{gather*} 
Using the fact that $\nabla_x G, \nabla_x L, \nabla_\alpha L$ are Lipschitz with respect to the measure argument uniformly with respect to the other variables, as well as Cauchy-Schwartz inequality in $\left(\mathcal{H}^T\right)^d$, we deduce that there exists a constant $C_{Lip}$ such that for any control $(\alpha_t)_{t\in [0,T]}\in \left(\mathcal{H}^T\right)^d$, 
\begin{gather*}
 \esp{\nabla_x G(X^{\alpha^n}_T,\mu^n_T)\cdot (X^{\alpha}_T-X^{\alpha^n}_T)+\int_0^T \left(\begin{array}{c} \nabla_x L(t,X^{\alpha^n}_t,\alpha^n_t,m^n_t)\\ \nabla_\alpha L(t,X^{\alpha^n}_t,\alpha^n_t,m^n_t) \end{array}\right)\cdot \left(\begin{array}{c} X^{\alpha}_t-X^{\alpha^n}_t\\ \alpha_t-\alpha^n_t\end{array}\right) dt}\\\geq -C_{Lip}(1+T^2)\|\alpha^n-\bar{\alpha}^n\|_T\|\alpha-\alpha^n\|_T
\end{gather*}
where 
\[\forall t\in [0,T],\quad \mu^n_t=\mathcal{L}(X^{\alpha^n}_t|\mathcal{F}^0_T), m^n_t=\mathcal{L}(X^{\alpha^n}_t,\alpha^n_t|\mathcal{F}^0_t).\]
Using this inequality for $\alpha=\alpha^*$, combining it with \eqref{eq: first order optimality mfc}, and using the displacement monotonicity of $G,L$ it follows that 
\[-c_L\|\alpha^*-\alpha^n\|^2_T\geq -C_{Lip}(1+T^2)\|\alpha^n-\bar{\alpha}^n\|_T\|\alpha^*-\alpha^n\|_T.\]
Consequently there exists a constant $C^*$ depending only on the Lipschitz constant of $\nabla_x G, \nabla_x L, \nabla_\alpha L$ with respect to the measure argument, $T$ and $c_L$ such that 
\[\|\alpha^*-\alpha^n\|_T\leq C^* \|\alpha^n-\bar{\alpha}^n\|_T.\]
In particular, 
\[\lim_{n\to \infty }\|\alpha^*-\alpha^n\|_T=0.\]
\end{proof}
Let us also remark that under appropriate convexity assumptions on the coefficients $L,G$, it is not necessary to consider potential MFGs. Indeed if we make the assumption that Hypothesis \ref{hyp: convexity in x} is in force and that the coefficients are flat monotone (a set of assumption which in general is stronger than displacement monotonicity), it is possible to show an analogous result to \cite{Perrin2020Fictitious} by proving that the exploitability associated to the game decreases at each step. The extension is quite straightforward in light of the arguments presented in this article and we do not make such developpement here. It is however unknown to us whether simple displacement monotonicity is sufficient to obtain convergence (as opposed to flat monotonicity+ convexity in the finite dimensional variables), and we leave such extension to the interested reader. 
\subsection{On common noise}
\subsubsection{Convergence without common noise}
Our main convergence result, Theorem \ref{thm: convergence to measure valued mfg} may appear very weak, especially compared to the convergence obtained in \cite{cardaliaguet-fictitious-play}. We claim that this is mainly due to the lack of compactness resulting from the common noise. Indeed, when there is no such noise, it is possible to obtain strong convergence results for controls under mild additional regularity assumptions. Although we do not prove this result in full detail let us briefly mention that the arguments used are quite similar to the ones we present in Section \ref{section: gradient descent}, namely the precompactness of the sequence $(\Lambda^n)_{n\in \mathbb{N}}$ and a stability result on FBSDEs. In this later section this stability follows from the non degeneracy of the idiosyncratic noise, whereas we can here rely on the monotonicity of coefficients following from Hypothesis \ref{hyp: convexity in x}, such stability result is quite standard, see for example \cite{extragradient-in-mfg} Remark 2.23.

\subsubsection{Other models of common noise}
\label{section: other common noise}
In this article we focused on the particular case of additive common noise. Although this model of noise is quite standard in MFG let us insist that our approach can be extended to more general types of common noises. Mainly we are interested in common noise affecting directly the coefficients, that is, we now consider coefficients 
\[\omega^0\mapsto (L(t,x,\alpha,m,\omega^0),G(x,\mu,\omega^0)).\]
In our opinion this kind of dependency is quite interesting from a modelling perspective since the noise now affects directly the environment of players through their cost functions. So long as the dependence on this noise is progressively measurable there is no particular difficulty in adapting our argument. Typical exemples include situations in which the cost of players depends on an additional noise process as described in \cite{meynard-1,meynard-2}
\begin{gather*}J(\alpha,(m_t)_{t\in [0,T]})=\esp{G(X^\alpha_T,p_T,\pi_d m_T)+\int_0^T L(t,X^\alpha_t,\alpha_t,p_t,m_t)dt},\\
    X^\alpha_t=X_0-\int_0^t \alpha_s ds+\sigma W_t,\\
    p_t=p_0+\int_0^t b(p_s)ds+\sigma^0W^0_t,
\end{gather*}
which are quite natural in many applications. Let us also insist that those models are not limited to a Brownian common noise but can be adapted to consider jump type common noises as well \cite{Bertucci-monotone}.
\subsection{Interpretation as a learning procedure}
\label{section: interpretation}
\subsubsection{Formation of equilibrium in mean field games}
In this section we present how the learning procedure we described in previous segments can be interpretated as a way players can accurately learn the control response of the crowd in repeated games, the state of the game thereby converging to a Nash equilibrium.  This idea comes from the seminal paper \cite{cardaliaguet-fictitious-play}, in which the authors proved agents can learn the distribution of other players. However let us stress the two major differences in our setting: First, players must learn not only the distribution of the other players but also that of their controls, moreover the averaging rule is quite different in our case since it takes place directly at the level of random variables. That is, players do not observe the distribution of other players directly but rather are able to observe the underlying random variable. Although this might seem less intuitive, we argue that there is a natural way of formulating this learning procedure. It does require to take a somewhat unusual interpretation of the MFG that is sometimes refered to as Lagrangian \cite{bertucci-livre}. In this setting, that we present formally and without common noise for now, the continuum of players is indexed by $\omega\in \Omega$. The state of a single player $\omega$ evolves according to 
\[X_t(\omega)=X_0(\omega)-\int_0^t \alpha_s(\omega) ds+\sigma W_t(\omega).\]
In this interpretation all players face a different realisation of the same Brownian motion $(W_t)_{t\in [0,T]}$. However they cannot observe directly observe this noise, which is why given a prior $(m_t)_{t\in [0,T]}$ for the distribution of all other players and their controls, they choose a control in feedback form $\alpha^*$ by minimizing the risk neutral cost 
\begin{gather}
    \label{eq: optimal control feedback}\alpha^*=\underset{(t,x)\mapsto \alpha(t,x)}{\text{arginf}} \esp{g(X^\alpha_T,\pi_d m_T)+\int_0^T L(t,X^\alpha_t,\alpha_t,m_t)dt},\\
  \nonumber  dX^\alpha_t=-\alpha(t,X^\alpha_t) dt+\sigma dW_t.
\end{gather}
Then, the evolution of the state of a single player $\omega$ is given by 
\[X_t(\omega)=X_0(\omega)-\int_0^t \alpha^*(s,X_s(\omega)) ds+\sigma W_t(\omega),\]
and its control by 
\begin{equation}
    \label{eq: control a single player}
    \forall s\in [0,T],\quad \alpha_s(\omega)=\alpha^*(s,X_s(\omega)).
\end{equation}
On the other hand, a player $\omega^*$ can observe the response and state of the other players 
\[\left(\left(X_t(\omega),\alpha_t(\omega)\right)_{t\in [0,T]}\right)_{\omega\neq \omega^*}.\]
We now have all the tools at hand to describe the learning procedure: Initially players all have the same belief on the actions of other players which we denotes $(\bar{X}^0_t(\omega),\bar{\alpha}^0_t(\omega))_{\omega\in \Omega}$. Moreover, each player $\omega$ is given a weight $d\mathbb{P}(\omega)$ in such fashion that players infer the distribution of other players through time to be 
\[m_t=\mathcal{L}^\mathbb{P}(\omega\mapsto (\bar{X}^0_t(\omega),\bar{\alpha}^0_t(\omega))).\]
It might seems arbitrary to give a different weight to different players. One possible interpretation is that the path $(X_t^\omega,\alpha^\omega_t)_{t\in [0,T]}$ is that of a representative player for all players facing the path $(X_0(\omega),(W_t(\omega))_{t\in [0,T]})$ and $d\mathbb{P}(\omega)$ indicates the proportion of players facing this realisation. Given the belief of players $(\bar{X}^n_t(\omega),\bar{\alpha}^n_t(\omega))_{\omega\in \Omega}$ at the stage $n+1$, an individual player $\omega^*$ choose its control according to the above rule \eqref{eq: control a single player}
\[t\mapsto (X^{n+1}_t(\omega^*),\alpha^{n+1}(t,X^{n+1}_t(\omega^*))),\]
where $\alpha^{n+1}$ is the optimal feedback associated to the control problem \eqref{eq: optimal control feedback} for the flow of measure given by the belief of players $\left(\mathcal{L}^\mathbb{P}\left( \omega\mapsto \bar{X}^n_t(\omega),\bar{\alpha}^n_t(\omega)\right)\right)_{t\in [0,T]}$.
Moreover they observes the path taken by other players 
\[\left(t\mapsto (X^{n+1}_t(\omega),\alpha^{n+1}(t,X^{n+1}_t(\omega))\right)_{\omega\neq \omega^*}.\]
Given those path, they update their belief directly on the trajectory of other players: they assume that at the next iteration the path followed by a player $\omega\neq \omega^*$ will be 
\[t\mapsto \eta_n(X^{n+1}_t(\omega),\alpha^{n+1}_t(\omega))+(1-\eta_n)(\bar{X}^{n}_t(\omega),\bar{\alpha}^{n}_t(\omega)),\]
where the pair $(\alpha^n_t,\bar{\alpha}^n_t)_{t\in [0,T]}$ indicates the sequence generated by the fictitious play procedure at step $n$ starting from the initial control $\bar{\alpha}^0$.
Having updated their belief on the trajectory followed by other players, individual players now infer that the distribution will be given by 
\[\bar{m}^{n+1}_t=\mathcal{L}^\mathbb{P}(\eta_n(X^{n+1}_t,\alpha^{n+1}_t)+(1-\eta_n)(\bar{X}^n_t,\bar{\alpha}^n_t)),\]
which is exactly the learning rule we proposed in this article. While the presentation of this learning rule is slightly technical (and not quite rigorous here), we believe that it is quite natural for players to learn the game by observing and trying to predict directly the actions of the other players, and not just the resulting distribution. In fact, in practise it is often easier for a player to observe the state and choices (controls) of their neighbour rather than the distribution of the position and controls of all players directly. Compared to \cite{cardaliaguet-fictitious-play}, the key takeaway, is we describe a learning procedure directly on the individual path taken by each player instead of on the global state of the system (distribution of players).

In the presence of common noise, the interpretation of the learning procedure is the same, except that players now average the actions of the crowd for each $\omega^0$ realisation of the common noise, that is, they update their belief according to 
\[(t,\omega^0)\mapsto \eta_n(X^{n+1}_t(\omega,\omega^0),\alpha^{n+1}_t(\omega,\omega^0))+(1-\eta_n)(\bar{X}^{n}_t(\omega,\omega^0),\bar{\alpha}^{n}_t(\omega,\omega^0)).\]
\begin{remarque}
Let us observe that in this approach, since all players admits the same feedback control $(t,x)\mapsto \alpha(t,x)$, this implies that the representative player problem
\begin{gather*}\inf_\beta \esp{g(X^\beta_T,\pi_d m_T)+\int_0^T L(t,X^\beta_t,\beta_t,m_t)dt},\\
    X^\beta_t=X_0-\int_0^t \beta_s ds+\sigma W_t.
\end{gather*}
is such that for its optimal control $\beta^*$ we have for any $\omega^*\in \Omega$
\[(X^{\beta^*}_t,\beta^*_t)(\omega^*)=(X_t(\omega^*),\alpha_t(\omega^*)).\]
Consequently the representative player process is exactly the process which to each player $\omega$ maps its optimal control and state.
\end{remarque}
\subsubsection{Black box learning of equilibria}
Clearly this learning procedure can be used to generate a sequence converging to a (measure valued) MFG equilibrium (Theorem \ref{thm: convergence to measure valued mfg}) and even to the unique Nash equilibrium in the monotone setting (Theorem \ref{thm: cv monotone setting}). However let us insist that no knowledge of the coefficients $L,G$ is required to carry out the learning procedure. If we possess a block box algorithm that given an priori $(m_t)_{t\in [0,T]}$ for the distribution of players and their controls yields optimal trajectories for this profile, then it is possible to use the proposed fictitious play to learn Nash equilibria without even having to estimate the coefficients. We believe this approach could be useful to learn from real world data or for agent learning in machine learning in which realised trajectories for the control and state of players may be available as data without prior knowledge of the optimisation problem considered. In this last case common noise can also be viewed as learning in a random environment and noise models as considered in Section \ref{section: other common noise} are quite natural.
\section{Gradient descent on the space of stochastic processes}
\label{section: gradient descent}
\subsection{Introduction}
In this section we now set $\sigma^0=0$, which is to say we consider dynamics without common noise, this is mainly a choice of convenience to avoid very technical details. However, we do give some insights on how to extend this idea to problems with common noise. While the first part of this article was dedicated to agent learning and formation of equilibria in mean field games, this second part concerns learning in mean field control directly. Namely, we present how under mild conditions on coefficients a principal agent trying to solve by gradient descent the mean field control problem
\[\inf_\alpha \left\{\int_0^T l(t,\mathcal{L}(X^\alpha_t,\alpha_t))dt+ g(\mathcal{L}(X^\alpha_T))\right\},\]
under the dynamics 
\[X^\alpha_t=X_0-\int_0^t \alpha_sds+\sigma W_t,\]
will eventually reach a local extremum of the potential $\mathcal{J}$, as it turns out, any such extremum is also a solution of the mean field game problem \eqref{eq: mfg equilibrium condition}. In consequence, this gradient descent algorithm on the space of controls yields an iterative procedure which converges to an equilibrium of the mean field game problem, hence the connection to learning controls in MFGs. To an extent, this also relates to the formation of equilibrium in MFG: a principal agent trying to minimize the potential $\mathcal{J}$ using only local gradient information will converge to a MFG equilibrium of the underlying game.
In this section we make the following assumptions
\begin{hyp}
    \label{hyp: reg mfc}
    The functions $l,g$ are continuously differentiable and with $L=\frac{\delta l}{\delta m}$, $G=\frac{\delta g}{\delta \mu}$ the functions $\nabla_x L,\nabla_\alpha L,\nabla_x G$ are Lipschitz in all arguments with constant $C_{Lip}$, for $\mathcal{W}_q$ with some $1\leq q<2$ in the measure argument. 
\end{hyp}
Those growth and regularity assumptions are in no way necessary and can be adapted to more general models (local regularity, power type growth...) provided that corresponding coercivity assumptions are imposed. In our case this takes the form

\begin{hyp}
    \label{hyp: coef mfc}
There exists a constant $C_b>0$ such that for any $(t,\mu,m)\in [0,T]\times\mathcal{P}_2(\reels^d)\times \mathcal{P}_2(\reels^{2d})$,
\[g(\mu)\geq -C_b,\]
\[l(t,m)\geq -C_b+\frac{1}{C_b}\int_{(x,a)\in \reels^{2d}} |a|^2 m(dx,da).\]
Moreover for any $(t,m)\in [0,T]\times \mathcal{P}(\reels^{2d})$, the function 
\[\alpha \mapsto L(t,x,\alpha,m),\]
is strongly convex uniformly in $x\in \reels^d$.
\end{hyp}
In particular under this assumption the Hamiltonian 
\[H(t,x,p,m)=\inf_\alpha \left\{L(t,x,\alpha,m)-\alpha\cdot p\right\},\]
is well defined, has Lipschitz gradient $\nabla_xH,\nabla_p H$ and is such that $p\mapsto \nabla_p H(t,x,p,m)$ is the inverse function of 
\[\alpha \mapsto \nabla_\alpha L(t,x,\alpha,m).\]
This is in fact the only reason we require $L$ to be convex in $\alpha$. Once again most results presented in this section could be adapted to more general settings, but we aim at presenting the key idea instead of making a very general claim.

\begin{lemma}
    Under the assumptions \ref{hyp: reg mfc} and \ref{hyp: coef mfc}, the function from $(\mathcal{H}^T)^d$ into $\reels$ 
    \[\alpha \mapsto \mathcal{J}(\alpha),\]
    admits a gradient on $\mathcal{H}^T$ which is the process given by 
\[\forall t\in [0,T],\quad \nabla\mathcal{J}(\alpha)_t=\nabla_\alpha L(X^\alpha_t,\alpha_t,\mathcal{L}(X^\alpha_t,\alpha_t))-U^\alpha_t,\]
where $(U^\alpha_t)_{t\in [0,T]}$ is defined through the BSDE
\begin{equation}
    \label{eq: bsde in gradient mfc}
    U^\alpha_t=\nabla_x G(X^\alpha_T,\mathcal{L}(X^\alpha_T))-\int_t^T \nabla_x L(X^\alpha_s,\alpha_s,\mathcal{L}(X^\alpha_s,\alpha_s))dt-\int_t^T Z^\alpha_sdW_s.\end{equation}
Moreover $\nabla \mathcal{J}$ is Lipschitz on $\left(\mathcal{H}^T\right)^d$ with a constant depending only on $C_{Lip}$ and $T$.
\end{lemma}
\begin{proof}
    Consider two controls $\alpha,\beta\in \mathcal{H}^T$, by standard computations already justified in the proof of Lemma \ref{lemma: first order optimality mfg}, we find that 
\begin{gather*}\lim_{\lambda \to 0}\left(\frac{\mathcal{J}(\alpha+\lambda (\beta-\alpha))-\mathcal{J}(\alpha)}{\lambda}\right)\\
=\esp{\nabla_x G(X^\alpha_T,\mathcal{L}(X^\alpha_T))\cdot (X^\beta_T-X^\alpha_T)}\\
+\esp{\int_0^T\left(\begin{array}{c} \nabla_x L(X^\alpha_t,\alpha_t,\mathcal{L}(X^\alpha_t,\alpha_t))\\ \nabla_\alpha L(X^\alpha_t,\alpha_t,\mathcal{L}(X^\alpha_t,\alpha_t)) \end{array}\right)\cdot \left(\begin{array}{c} X^{\beta}_t-X^{\alpha}_t\\ \beta_t-\alpha_t\end{array}\right) dt}
\end{gather*}
We now introduce $(U^\alpha_t)_{t\in [0,T]}$ solution to the BSDE \eqref{eq: bsde in gradient mfc}, which has a unique strong solution since by assumption 
\[\esp{\int_0^T |\alpha_s|^2ds}<+\infty,\]
and coefficients are Lipschitz \cite{Zhang2017BSDE}. By an application of Ito's Lemma we find that 
\begin{gather*}\lim_{\lambda \to 0}\left(\frac{\mathcal{J}(\alpha+\lambda (\beta-\alpha))-\mathcal{J}(\alpha)}{\lambda}\right)\\
=\esp{\int_0^T \left(\nabla_\alpha L(X^\alpha_t,\alpha_t,\mathcal{L}(X^\alpha_t,\alpha_t))-U^\alpha_t\right)\cdot(\beta_t-\alpha_t)dt }\\
=\langle \nabla \mathcal{J}(\alpha),\beta-\alpha\rangle^T,
\end{gather*}
where $\nabla\mathcal{J}(\alpha)$ is the process in $\mathcal{H}^T$ defined by 
\[\forall t\in [0,T],\quad \nabla\mathcal{J}(\alpha)_t=\nabla_\alpha L(X^\alpha_t,\alpha_t,\mathcal{L}(X^\alpha_t,\alpha_t))-U^\alpha_t.\]
Moreover, we check easily that under Hypothesis \ref{hyp: reg mfc}, this gradient is Lipschitz on $\mathcal{H}^T$. Indeed for any two controls $\alpha,\beta \in \mathcal{H}^T$, 
\begin{gather*}\langle \nabla \mathcal{J}(\alpha)-\nabla \mathcal{J}(\beta),\beta-\alpha\rangle^T\\
=\esp{(\nabla_x G(X^\alpha_T,\mathcal{L}(X^\alpha_T))-\nabla_x G(X^\beta_T,\mathcal{L}(X^\beta_T)))\cdot (X^\beta_T-X^\alpha_T)}\\
+\esp{\int_0^T\left(\begin{array}{c} \nabla_x L(X^\alpha_t,\alpha_t,\mathcal{L}(X^\alpha_t))-\nabla_x L(X^\beta_t,\beta_t,\mathcal{L}(X^\beta_t,\beta_t))\\ \nabla_\alpha L(X^\alpha_t,\alpha_t,\mathcal{L}(X^\alpha_t,\alpha_t)) -\nabla_\alpha L(X^\beta_t,\beta_t,\mathcal{L}(X^\beta_t,\beta_t)) \end{array}\right)\cdot \left(\begin{array}{c} X^{\beta}_t-X^{\alpha}_t\\ \beta_t-\alpha_t\end{array}\right) dt}\\
\leq C_{Lip}(1+T+\sqrt{T})\|\alpha-\beta\|_T^2,
\end{gather*}
where we have used, Jensen and Cauchy-Schwartz inequality in $\mathcal{H}^T$.
\end{proof}
Let us remark that this gradient structure is very much related to the idea of reformulating mean field games as variational inequalities on the space of stochastic processes introduced in \cite{extragradient-in-mfg}. It is well known that extragradient methods can (very formally) be seen as a generalization of gradient descent for non potential variational inequalities. In the potential framework (mean field control), we observe the same phenomenon for MFGs. 
\subsection{Gradient descent}
We now investigate the associated gradient descent on $\mathcal{H}^T$
\begin{equation}
    \label{eq: gradient descent}
    \alpha^{n+1}=\alpha^n-\eta_n \nabla \mathcal{J}(\alpha^n),\end{equation}
with $(\eta_n)_{n\in \mathbb{N}}\subset [0,1]$ a non increasing sequence of step sizes. We first prove that along the proposed descent, the potential $\mathcal{J}$ is decreasing 
\begin{lemma}
\label{lemma: iteration}
Under the update rule \eqref{eq: gradient descent}, for any initial condition $\alpha^0\in \mathcal{H}^T$, the following holds 
\begin{equation}
\label{eq: iteration}
\mathcal{J}(\alpha^{n+1})-\mathcal{J}(\alpha^n)\leq \left(\|\nabla \mathcal{J}\|_{Lip}\eta_n^2-\eta_n\right)\|\nabla\mathcal{J}(\alpha^n)\|^2_T.\end{equation}
In particular if \[\lim_{n\to \infty}\eta_n=0,\]
then the sequence $(\alpha^n)_{n\in \mathbb{N}}$ is uniformly bounded in $\mathcal{H}^T$ by a constant depending only on $C_{Lip},T,\alpha^0$ and $(\eta_n)_{n\in \mathbb{N}}$.
\end{lemma}
\begin{proof}
By standard Gronwall type estimates, it is straightforward that if $\alpha^0\in \left(\mathcal{H}^T\right)^d$ then,  for any $n\in \mathbb{N}$, $\alpha^n\in \mathcal{H}^T$. Fixing $n\in \mathbb{N}$, we observe that 
\begin{align*}
\mathcal{J}(\alpha^{n+1})-\mathcal{J}(\alpha^n)&=\int_0^1\langle \nabla \mathcal{J}(\lambda \alpha^{n+1}+(1-\lambda) \alpha^n),\alpha^{n+1}-\alpha^n\rangle^Td\lambda\\
&\leq \langle \nabla \mathcal{J}(\alpha^n),\alpha^{n+1}-\alpha^n\rangle^T+\|\nabla\mathcal{J}\|_{Lip}\|\alpha^{n+1}-\alpha^n\|^2.
\end{align*}
Since 
\[\alpha^{n+1}-\alpha^n=-\eta_n \nabla \mathcal{J}(\alpha^n),\]
we find that 
\[\mathcal{J}(\alpha^{n+1})-\mathcal{J}(\alpha^n)\leq \left(\|\nabla \mathcal{J}\|_{Lip}\eta_n^2-\eta_n\right)\|\nabla\mathcal{J}(\alpha^n)\|^2_T.\]
As for the second statement, if $(\eta_n)_{n\in \mathbb{N}}$ tends to $0$, then there exists a $n_0\in \mathbb{N}$ such that 
\[\forall n\geq n_0,\quad \|\nabla \mathcal{J}\|_{Lip}\eta_n^2-\eta_n\leq 0.\]
It follows that for any such $n$
\[\mathcal{J}(\alpha^n)\leq \mathcal{J}(\alpha^0)+\|\nabla\mathcal{J}\|_{Lip}\sum_{i=1}^{n_0}\eta_i^2 \|\nabla\mathcal{J}(\alpha^i)\|_T^2.\]
By the growth of coefficients and Gronwall Lemma the righthand side is bounded. On the other hand by the coercivity of $l$ and the boundedness from below of $g$ we find that 
\[\mathcal{J}(\alpha^n)\geq -(1+T)C_b+\frac{1}{C_b}\|\alpha^n\|^2_T,\]
which yields the uniform bound in $\mathcal{H}^T$ on the sequence $(\alpha^n)_{n\in \mathbb{N}}$.
\end{proof}
In consequence we deduce that 
\begin{lemma}
\label{lemma: convergence of gradient}
Let $(\eta_n)_{n\in \mathbb{N}}$ satisfies 
\begin{equation}
\label{eq: condition eta}\sum_{n=1}^{+\infty} \eta_n=+\infty,\quad \sum_{n=1}^{+\infty}\eta_n^2<+\infty,\end{equation}
then for any $\alpha^0\in \mathcal{H}^T$, the sequence $(\alpha^n)_{n\in \mathbb{N}}$ generated by \eqref{eq: gradient descent} is bounded in $\mathcal{H}^T$ and 
\[\lim_{n\to \infty} \|\nabla \mathcal{J}(\alpha^n)\|_T=0.\]
\end{lemma}
\begin{proof}
First since 
\[\sum_{n=1}^{+\infty}\eta_n^2<+\infty,\]
we deduce that 
\[\lim_{n\to \infty}\eta_n=0,\]
and that we may apply Lemma \ref{lemma: iteration}. Since the sequence $(\alpha^n)_{n\in \mathbb{N}}$ is bounded in $\mathcal{H}^T$, it follows that $(\|\nabla\mathcal{J}(\alpha^n)\|^2_T)_{n\in \mathbb{N}}$ is bounded in $\reels^+$. In particular we deduce from the inequality \eqref{eq: iteration} that there exists a constant $C$ such that for any $n\in \mathbb{N}$
\[\mathcal{J}(\alpha^{n+1})-\mathcal{J}(\alpha^n)\leq \left(\|\nabla \mathcal{J}\|_{Lip}\eta_n^2-\eta_n\right)\|\nabla\mathcal{J}(\alpha^n)\|^2_T.\]
\[\sum_{k=0}^n \eta_k\|\nabla\mathcal{J}(\alpha^k)\|^2_T\leq C\sum_{k=0}^n \eta_k^2+\mathcal{J}(\alpha^0)-\mathcal{J}(\alpha^{n+1}).\]
Since the potential $\mathcal{J}$ is bounded from bellow we finally deduce that 
\[\sum_{n\in \mathbb{N}} \eta_n\|\nabla\mathcal{J}(\alpha^n)\|^2_T<+\infty.\]
On the other hand let us also remark that there exists a constant $C$ such that for any $n\in \mathbb{N}$
\[\left|\|\mathcal{J}(\alpha^{n+1})\|^2_T-\|\mathcal{J}(\alpha^n)\|^2_T\right|\leq C\eta_n,\]
this follows once again from the boundedness of $(\|\nabla\mathcal{J}(\alpha^n)\|^2_T)_{n\in \mathbb{N}}$ and the fact that $\nabla \mathcal{J}$ is Lipschitz. In light of those inequalities we conclude thanks to Lemma \ref{lemma: step size sequence} that 
\[\lim_{n\to \infty} \|\nabla \mathcal{J}(\alpha^n)\|^2_T=0.\]
\end{proof}
Let us insist that the above Lemma still hold in the presence of common noise (i.e. even in this case the gradient of $\mathcal{J}$ tends to $0$ along the sequence $(\alpha^n)_{n\in \mathbb{N}})$. However due to the lack of compactness in the space of controls $\left(\mathcal{H}^T\right)^d$, this is not sufficient to conclude that the sequence $(\alpha^n)_{n\in \mathbb{N}}$ converges to a MFG equilibrium. Instead we rely once again on the theory of relaxed controls to recover some compactness on the problem. Although we believe it should be possible to prove that in the presence of common noise the convergence holds in the sense of Definition \ref{def: measure valued mfg} by looking at the associated cost 
\[\esp{g(X^{\alpha^n}_T,\mathcal{L}(X^{\alpha^n}_T|\mathcal{F}^0_T))+\int_0^TL(t,X^{\alpha^n}_t,\alpha^n_t,\mathcal{L}(X^{\alpha^n}_t,\alpha^n_t|\mathcal{F}^0_t))},\]
and using the fact that $\alpha^n$ is almost a solution of finite dimensional optimal control problem (since $\|\nabla \mathcal{J}(\alpha^n)\|$ is small), we do not carry this technical extension here and focus on the straightforward case of $\sigma^0=0$.
\begin{thm}
    \label{thm: convergence gradient descent}
Under hypotheses \ref{hyp: coef mfc} and \ref{hyp: reg mfc}, if 
\[\sigma>0,\]
then for any sequence $(\eta_n)_{n\in \mathbb{N}}\subset [0,1]$ satisfying \eqref{eq: condition eta}, the sequence $(\alpha^n)_{n\in \mathbb{N}}$ generated by \eqref{eq: gradient descent} is precompact in $\left(\mathcal{H}^T\right)^d$ and any limit point $\alpha^*$ is a solution of the mean field game problem \eqref{eq: mfg equilibrium condition}.
\end{thm}
\begin{proof}
First, by Lemma \ref{lemma: convergence of gradient}, the sequence $(\alpha^n)_{n\in \mathbb{N}}$ is bounded in $\left(\mathcal{H}^T\right)^d$. In consequence we deduce that the sequence $(\mu^n_T,\Lambda^n)_{n\in \mathbb{N}}$ defined for fixed $n\in \mathbb{N}$ by 
\[\mu^n_T=\mathcal{L}(X^{\alpha^n}_T),\]
and 
\[\Lambda^n(dt,dm)=\delta_{\mathcal{L}(X^{\alpha^n}_t,\alpha^n_t)}(dm)dt,\]
is precompact in $\mathcal{P}_q(\reels^d)\times \mathcal{P}([0,T]\times \mathcal{P}_q(\reels^{2d}))$ for any $1\leq q<2$. Up to the extraction of a subsequence, let us assume that $(\mu^n_T,\Lambda^n)_{n\in \mathbb{N}}$ converges to some limit $(\mu^*_T,\Lambda^*)$. We are going to show that there exists a unique strong solution $(X_t,U_t,Z_t)_{t\in [0,T]}$ to the FBSDE 
\begin{equation}
\label{eq: fbsde limit mfc}
\left\{
\begin{array}{l}
    X_t=\displaystyle X_0-\int_0^t\int_{m\in \mathcal{P}(\reels^{2d})} \nabla_p H(s,X_s,U_s,m)\Lambda^*(ds,dm)+\sigma W_t,\\
    \displaystyle U_t=\nabla_x G(X_T,\mu^*_T)+\int_t^T \int_{m\in \mathcal{P}(\reels^{2d})}\nabla_x  H(s,X_s,U_s,m)\Lambda^*(ds,dm)-\int_t^T Z_sdW_s,\\
\end{array}
\right.
\end{equation}
and that the couple $(X^n_t,U^n_t)_{t\in [0,T]}$ defined through
\begin{equation}
\left\{
\begin{array}{l}
    X^n_t=\displaystyle X_0-\int_0^t \alpha^n_sds+\sigma W_t,\\
    \displaystyle U^n_t=\nabla_x G(X^n_T,\mathcal{L}(X^n_T))-\int_t^T \nabla_x L(X^n_s,\alpha^n_s,\mathcal{L}(X^n_s,\alpha^n_s))ds-\int_t^T Z^n_sdW_s\\
\end{array}
\right.
\end{equation}
converges to $(X_t,U_t)_{t\in [0,T]}$ in $\left(\mathcal{H}^T\right)^{2d}$. The first point follows from the fact that the FBSDE \eqref{eq: fbsde limit mfc} admits a Lipschitz decoupling field which formally solves 
\[\partial_t U+F(t,x,U)\cdot \nabla_x U-\frac{1}{2}\sigma^2 \Delta_x U=G(t,x,U),\]
where the functions $F,G$ are defined by 
\[F(t,x,u)=\int_{m\in \mathcal{P}(\reels^{2d})}\nabla_p H(t,x,u,m)\Lambda_t(dm), \quad G(t,x,u)=-\int_{m\in \mathcal{P}(\reels^{2d})}\nabla_x H(t,x,u,m)\Lambda_t(dm),\]
for $(\Lambda_t)_{t\in [0,T]}$ defined by the disintegration of $\Lambda$ as 
\[\Lambda(ds,dm)=\Lambda_s(dm)ds.\]
On the other hand, by definition of $\nabla\mathcal{J}(\alpha^n)$ and the invertibility of $\nabla_\alpha L$,
\begin{equation}
    \label{eq: alpha in function of u}
    \forall t\in [0,T],\quad \alpha^n_t=\nabla_p H(X^n_t,U^n_t+\nabla \mathcal{J}(\alpha^n)_t,\mathcal{L}(X^n_t,\alpha^n_t)),
\end{equation}
and let us observe that the triple $(X^n_t,U^n_t,Z^n_t)_{t\in [0,T]}$ solves 
\begin{equation}
\label{eq: fbsde reformulation}
\left\{
\begin{array}{l}
    X^n_t=\displaystyle X_0-\int_0^t \int_{m\in \mathcal{P}(\reels^{2d})}\nabla_p H(X^n_s,U^n_s+\nabla \mathcal{J}(\alpha^n)_s,m)\Lambda^n(ds,dm)+\sigma W_t,\\
    \displaystyle U^n_t=\nabla_x G(X^n_T,\mu^n_T)+\int_t^T \int_{m\in \reels^{2d}}\nabla_x  H(X^n_s,U^n_s+\nabla \mathcal{J}(\alpha^n)_s,m)\Lambda^n(ds,dm)-\int_t^T Z^n_sdW_s\\
\end{array}
\right.
\end{equation}
Clearly we want to use a stability argument on the FBSDE system \eqref{eq: fbsde limit mfc}, which should follow from the regularity of its decoupling field $U$. Since $U$ is not sufficiently smooth to justify directly the PDE formulation, we introduce $(F^\varepsilon,G^\varepsilon)_{\varepsilon>0}$ a regularization of $(F,G)$ in the time variable by convolution with a smooth compactly supported kernel, in such Fashion that $F^\varepsilon,G^\varepsilon$ have the same Lipschitz constant as $F,G$ in $(x,u)$ and letting
\[\omega_{coef}(\varepsilon)=\sup_{(x,u)\in \reels^{2d}}\int_0^T\left|\frac{F^\varepsilon(t,x,u)-F(t,x,u)}{1+|x|+|u|}\right|^2dt+ \int_0^T\left|\frac{G^\varepsilon(t,x,u)-G(t,x,u)}{1+|x|+|u|}\right|^2dt,\]
we have 
\[\lim_{\varepsilon\to 0} \omega_{coef}(\varepsilon)=0,\]
which follows from the fact that $\frac{F}{1+|x|+|u|},\frac{G}{1+|x|+|u|}\in L^2([0,T], L^\infty( \reels^{2d},\reels^d))$

We consider the sequence of solutions $(U^\varepsilon)_{\varepsilon>0}$ to 
\[\left\{\begin{array}{cc}
    \partial_t U^\varepsilon+F^\varepsilon(t,x,U^\varepsilon)\cdot \nabla_x U^\varepsilon-\frac{1}{2}\sigma^2\Delta U^\varepsilon=G^\varepsilon(t,x,U^\varepsilon), \quad (t,x)\in (0,T)\times \reels^d,\\
    U^\varepsilon(T,x)=\nabla_x G(x,\mu^*_T),\quad x\in \reels^d,
\end{array}\right.\]
By standard parabolic estimates \cite{parabolic-estimate-on-whole-space}, for any $\varepsilon>0$, $U^\varepsilon\in C^{1,2}([0,T)\times \reels^d,\reels^d)$ and is Lipschitz in $x$ with a constant independent of $\varepsilon$. Moreover $U^\varepsilon$ has appropriate growth assumptions on its derivatives to apply Ito's Lemma, and we deduce that for any $n\in \mathbb{N}$
\begin{align*}\nabla_x G(X^n_T,\mu^*_T)-U^\varepsilon(t,X^{\alpha^n}_t)&=\int_t^T G^\varepsilon(s,X^n_s,U^\varepsilon(s,X^n_s))ds\\&-\int_t^T \left(F^n(s,X^n_s,U^n_s)-F^\varepsilon(s,X^n_s,U^\varepsilon(s,X^n_s))\right)\cdot \nabla_x U^\varepsilon(s,X^n_s)ds,\end{align*}
with the notation 
\[F^n(t,x,u)=\int_{m\in \mathcal{P}(\reels^{2d})}\nabla_p H(t,x,u,m)\Lambda^n_t(dm).\]
Introducing the error 
\begin{align*}
\omega_U(\varepsilon)=\sup_{(t,x)\in [0,T]\times \reels^{d}}\frac{|U^\varepsilon(t,x)-U(t,x)|}{1+|x|},
\end{align*}
that tends to $0$ as $\varepsilon$ tends to $0$ by Azerla-Ascoli, and the stability of Lipschitz solution to parabolic transport systems with Lipschitz coefficients \cite{Lipschitz-solutions}, we deduce that for any $\varepsilon>0$ and any $1\leq q<2$
\begin{align*}\esp{|U^n_t-U(t,X^n_t)|^q}&\leq C_{growth}\left(\omega_U(\varepsilon)+\omega_{coef}(\varepsilon)\right)\left(1+\int_t^T\esp{|X^n_s|^2+|U^n_s|^2}ds\right)\\
&+2\|\nabla_x U\|_\infty\esp{\int_t^T \left|F^n(s,X^n_s,U^n_s+\nabla \mathcal{J}(\alpha^n)_s)-F(s,X^n_s,U(s,X^n_s))\right|^q ds}\\&+2\esp{\int_t^T\left|G^n(s,X^n_s,U^n_s+\nabla \mathcal{J}(\alpha^n)_s)-G(s,X^n_s,U(s,X^n_s))\right|^qds }\\
&+\esp{|\nabla_x G(X^n_T,\mu^n_T)-\nabla_x G(X^n_T,\mu^*_T)|^q},
\end{align*}
for a constant $C_{growth}$ depending on the linear growth of $U,\nabla_x H,\nabla_p H$. Letting $\varepsilon$ tends to $0$ and using the regularity of $H$ from Hypothesis \ref{hyp: reg mfc} we deduce that there exists a constant $C_{Lip}$ depending on $\sigma$ and the Lipschitz constant of coefficients and a $1\leq q <2$ such that, 
\[\esp{|U^n_t-U(t,X^n_t)|^q}\leq C_{Lip}\left(\underbrace{\mathcal{W}_q^q(\mu^n_T,\mu^*_T)+\mathcal{W}_q^q(\Lambda^n,\Lambda^*)+\|\nabla \mathcal{J}(\alpha^n)\|^q_T}_{\varepsilon_n}+\int_t^T\esp{|U^n_s-U(s,X^n_s)|^q}ds\right).\]
By Gronwall Lemma, it follows that 
\[\exists C>0, \quad \forall t\in [0,T]\quad \esp{|U^n_t-U(t,X^n_t)|^q}\leq C\varepsilon_n,\]
moreover, by assumption on the sequence $(\mu^n_T,\Lambda^n)_{n\in \mathbb{N}}$ and Lemma \ref{lemma: convergence of gradient} 
\[\lim_{n\to \infty} \varepsilon_n=0.\]
Using this estimate we deduce once again by Gronwall Lemma that 
\[\exists C>0, \quad \forall t\in [0,T]\quad \esp{|X^n_t-X_t|^q}\leq C\varepsilon_n,\]
in particular we conclude to 
\[\lim_{n\to \infty} \esp{|X^n_t-X_t|^q}+\esp{|U^n_t-U_t|^q}=0.\]
From the representation \eqref{eq: alpha in function of u}, 
\begin{align*}\esp{\int_0^T |\alpha^n_t-F(t,X_t,U_t)|^qdt}&\leq \|\nabla_p H\|_{Lip}\left(\int_0^T\left(\esp{|X^n_t-X_t|^q}+\esp{|U^n_t-U_t|^q}\right)dt+\mathcal{W}_q(\Lambda^n,\Lambda)+\|\nabla \mathcal{J}(\alpha^n)\|_T^q\right).\end{align*}
Introducing $\alpha^*\in \left(\mathcal{H}^T\right)^d$ defined by 
\[\forall t\in [0,T],\quad \alpha^*_t=F(t,X_t,U_t),\]
it follows that
\[\lim_{n\to \infty}\esp{\int_0^T |\alpha^n_t-\alpha^*_t|^qdt}=0,\]
which also leads to 
\[\lim_{n\to \infty}\mathcal{W}_q\left(\Lambda^n, \delta_{\mathcal{L}(X_t,\alpha^*_t)}(dm)dt\right),\]
and finally to 
\[\int_0^T \mathcal{W}_q(\Lambda^*_t,\mathcal{L}(X_t,\alpha^*_t))dt=0.\]
This means that 
\[\esp{\int_0^T \left|\nabla_\alpha L(X_t,\alpha^*_t,\mathcal{L}(X_t,\alpha^*_t))-U_t\right|^2dt}=0,\]
and the triple $(X_t,U_t,Z_t)_{t\in [0,T]}$ is a solution to 
\begin{equation*}
\left\{
\begin{array}{l}
    X_t=\displaystyle X_0-\int_0^t\nabla_p H(s,X_s,U_s,\mathcal{L}(X_s,\alpha^*_s))ds+\sigma W_t,\\
    \displaystyle U_t=\nabla_x G(X_T,\mathcal{L}(X_T))+\int_t^T \nabla_x  H(s,X_s,U_s,\mathcal{L}(X_s,\alpha^*_s))ds-\int_t^T Z_sdW_s,\\
\end{array}
\right.
\end{equation*}
From the regularity of coefficients and the fact that $\sigma>0$, given $(\mathcal{L}(X_s,\alpha^*_s))_{s\in [0,T]}$, there exists a unique solution to the (finite dimensional) $FBSDE$ with coefficients 
\[(t,x,u)\mapsto (-\nabla_p H,-\nabla_x H)(t,x,u,\mathcal{L}(X_t,\alpha^*_t)),\]
and we deduce that $\alpha^*$ is a solution to the MFG \eqref{eq: mfg equilibrium condition}. Since we worked along a general converging subsequence of $(\mu^n_T,\Lambda^n)_{n\in \mathbb{N}}$ we deduce that this result hold along any such subsequence. To extract a converging subsequence of $(\alpha^n)_{n\in \mathbb{N}}$ it is sufficient to extract a converging subsequence $(\mu^{\varphi(n)}_T,\Lambda^{\varphi(n)})_{n\in \mathbb{N}}$ of $(\mu^n_T,\Lambda^n)_{n\in \mathbb{N}}$, and to consider the sequence $(\alpha^{\varphi(n)})_{n\in \mathbb{N}}$, hence the precompactness of $(\alpha^n)_{n\in \mathbb{N}}$, the fact that any limit point is a MFG equilibrium is what we just demonstrated.
\end{proof}

\subsection{Iterative scheme and numerical concerns}
Equivalently, the gradient descent can be rewritten as 
\begin{equation}
\label{eq: gradient descent scheme}
\forall t\in [0,T],\quad 
\left\{
    \begin{array}{l}
       \displaystyle  X^n_t=X_0-\int_0^t \alpha^n_sds+\sigma W_t,\\
        \displaystyle U^n_t=\nabla_x G(X^n_T,\mathcal{L}(X^n_T))-\int_t^T \nabla_x L(s,X^n_s,\alpha^n_s,\mathcal{L}(X^n_s,\alpha^n_s))ds-\int_t^TZ^n_s dW_s,\\
         \alpha^{n+1}_t=\alpha^n_t-\eta_n \left(\nabla_\alpha L(X^n_t,\alpha^n_t,\mathcal{L}(X^n_t,\alpha^n_t))-U^n_t\right).
    \end{array}
\right.
\end{equation}
By Theorem \ref{thm: convergence gradient descent}, for this iterative scheme on decoupled systems of FBSDEs, the sequence $(\alpha^n)_{n\in \mathbb{N}}$ is precompact in $\left(\mathcal{H}^T\right)^d$ and any converging subsequence has for limit a MFG equilibrium. In practise the backward process $(U^n_t)_{t\in [0,T]}$ can be computed by approximating the conditional expectation with a regression step and this should lead to a rather straightforward numerical implementation. However there is one technical details in that at the n-th step, $(U^n_t)_{t\in [0,T]}$ is only a Markovian process in function of $(V^n_t)_{t\in [0,T]}=(\alpha^0_t,X^1_t,U^1_t,\cdots, X^{n-1}_t,U^{n-1}_t,X^n_t)_{t\in [0,T]}$ and not just $(X^n_t)_{t\in [0,T]}$, consequently to compute the conditional expectation with respect to the filtration $\mathcal{F}_t$, the regression step should be taken against $V^n_t$. Clearly this is not a viable numerical method, one way to avoid this issue is by computing directly the conditional expectation with respect to the family $(W_s)_{s\leq t}$ by a Wiener Chaos expension \cite{Briand_Labart_2014} but this is still computationally intensive. Instead we here argue that the conditional expectation can be accurately approximated by a regression step with respect to $X^n_t$ only. Let us remind that we have already established in the proof of Theorem \ref{thm: convergence gradient descent}, that there exists a constant $C>0$ such that along any converging subsequence of $(\alpha^n)_{n\in \mathbb{N}}$, there exists a Lipschitz function $U$ and a $1\leq q<2$, such that 
\[\esp{|U^n_t-U(t,X^n_t)|^q}\leq C\varepsilon_n,\]
with 
\[\lim_{n\to \infty} \varepsilon_n=0.\]
In consequence, for such $q$, 
\begin{align*}\esp{\left|\espcond{U^n_t}{X^n_t}-U^n_t\right|^q}&=\esp{|\left|\espcond{U^n_t-U(t,X^n_t)}{X^n_t}-(U^n_t-U(t,X^n_t))\right|^q}\\
    &\leq 2C \varepsilon_n.
\end{align*}
While this introduce another layer of error into the scheme, we believe the tradeof in terms of speed for each iteration to be worth it. Although we do not fully analyze the discretized scheme and its convergence in this article, let us mention that similar ideas have already been used for FBSDEs \cite{Bender_Zhang_2008}.

\section*{acknowledgements}
\begin{center}
The author is thankful to François Delarue and Nizar Touzi for helpful discussions and comments on the manuscript. 

\quad

Charles Meynard acknowledge the financial support of the European Research Council (ERC) under the European Union’s Horizon Europe research and innovation program (ELISA project, Grant agreement No. 101054746). Views and opinions expressed are however those of the author only and do not necessarily reflect those of the European Union or the European Research Council Executive Agency. Neither the European Union nor the granting authority can be held responsible for them
\end{center}

\bibliographystyle{plain}
\bibliography{source}

\end{document}